\documentclass[12pt]{article}

\usepackage[T1]{fontenc}
\usepackage[cp1250]{inputenc}

\usepackage{amsmath}
\usepackage{amsfonts}
\usepackage{amssymb}
\usepackage{amsthm}
\usepackage{amscd}
\usepackage{wasysym}
\input xy
\xyoption{all}

\newcommand{\N}{\mathbb N}

\newcommand{\Zet}{\mathbb Z}

\newcommand{\Real}{\mathbb R}

\newcommand{\Dbf}{\mathbf D}
\newcommand{\Abf}{\mathbf A}
\newcommand{\Vbf}{\mathbf V}
\def\fbr#1#2#3{#1\times_{#3}#2}
\def\aff{\text{aff}}
\def\Aff{\text{Aff}}

\def\Sec{\text{Sec}}
\def\Vect{\text{Vect}}

\def\mlt#1{\cdot_#1}

\def\Vmodel{\mathcal V}
\def\dblVmodel{\mathcal V\!\mathcal V}

\def\vhull#1{\widehat{#1}}
\def\dblvhull#1{\widehat{\widehat{#1}}}

\def\und#1{\underline{#1}}
\def\dund#1{\underline{\underline{#1}}}

\def\spec#1{\mathbf #1}

\def\affdual#1{#1^{\#}}
\def\vectdual#1{#1^{\star}}

\def\dblvecdual#1#2{#1^{\varhexstar #2}}
\def\dblaffdual#1#2{#1^{\# #2}}

\def\tgT{\sT}
\def\ctgT{\sT^\star}
\def\Phase{\sP}
\def\Contact{\sC}
\def\dualP{\sP^\bullet}
\def\redtgT{\tilde{\sT}}
\def\redctgT{\tilde{\sT}^\star}
\def\redtgTbar{\bar{\sT}}

\def\ahyp#1{#1^\ddag}

\def\affdiff{\textbf d} %affine differential
\def\cdiff{\textbf c}

\def\subaff#1#2{#1^{\natural #2}}

\def\afftg#1{\mathbb S #1}
\def\affctg#1{\mathbb S^\bullet #1}

\def\dblB{\mathbb B}

\newdir{|>}{!/4,5pt/@{|}*:(1,-.2)@^{>}*:(1,+.2)@_{>}}

\font\black=cmbx10 \font\sblack=cmbx7 \font\ssblack=cmbx5 \font\blackital=cmmib10 \skewchar\blackital='177
\font\sblackital=cmmib7 \skewchar\sblackital='177 \font\ssblackital=cmmib5 \skewchar\ssblackital='177
\font\sanss=cmss12 \font\ssanss=cmss8 scaled 900 \font\sssanss=cmss8 scaled 600 \font\blackboard=msbm10
\font\sblackboard=msbm7 \font\ssblackboard=msbm5 \font\caligr=eusm10 \font\scaligr=eusm7 \font\sscaligr=eusm5
 \font\fraktur=eufm10 \font\sfraktur=eufm7 \font\ssfraktur=eufm5 

\usepackage{color,soul}
\font\bsymb=cmsy10 scaled\magstep2
\def\all#1{\setbox0=\hbox{\lower1.5pt\hbox{\bsymb
       \char"38}}\setbox1=\hbox{$_{#1}$} \box0\lower2pt\box1\;}
\def\exi#1{\setbox0=\hbox{\lower1.5pt\hbox{\bsymb \char"39}}
       \setbox1=\hbox{$_{#1}$} \box0\lower2pt\box1\;}

\def\tx#1{{\fam0\relax#1}}

\newfam\bifam
\textfont\bifam=\blackital \scriptfont\bifam=\sblackital \scriptscriptfont\bifam=\ssblackital

\newfam\blfam
\textfont\blfam=\black \scriptfont\blfam=\sblack \scriptscriptfont\blfam=\ssblack

\newfam\bbfam
\textfont\bbfam=\blackboard \scriptfont\bbfam=\sblackboard \scriptscriptfont\bbfam=\ssblackboard

\newfam\ssfam
\textfont\ssfam=\sanss \scriptfont\ssfam=\ssanss \scriptscriptfont\ssfam=\sssanss
\def\ss#1{{\fam\ssfam\relax#1}}

\newfam\clfam
\textfont\clfam=\caligr \scriptfont\clfam=\scaligr \scriptscriptfont\clfam=\sscaligr

\newfam\frfam
\textfont\frfam=\fraktur \scriptfont\frfam=\sfraktur \scriptscriptfont\frfam=\ssfraktur

\def\hpb#1{\setbox0=\hbox{${#1}$}
    \copy0 \kern-\wd0 \kern.2pt \box0}
\def\vpb#1{\setbox0=\hbox{${#1}$}
    \copy0 \kern-\wd0 \raise.08pt \box0}

\def\pmb#1{\setbox0\hbox{${#1}$} \copy0 \kern-\wd0 \kern.2pt \box0}
\def\pmbb#1{\setbox0\hbox{${#1}$} \copy0 \kern-\wd0
      \kern.2pt \copy0 \kern-\wd0 \kern.2pt \box0}
\def\pmbbb#1{\setbox0\hbox{${#1}$} \copy0 \kern-\wd0
      \kern.2pt \copy0 \kern-\wd0 \kern.2pt
    \copy0 \kern-\wd0 \kern.2pt \box0}
\def\pmxb#1{\setbox0\hbox{${#1}$} \copy0 \kern-\wd0
      \kern.2pt \copy0 \kern-\wd0 \kern.2pt
      \copy0 \kern-\wd0 \kern.2pt \copy0 \kern-\wd0 \kern.2pt \box0}
\def\pmxbb#1{\setbox0\hbox{${#1}$} \copy0 \kern-\wd0 \kern.2pt
      \copy0 \kern-\wd0 \kern.2pt
      \copy0 \kern-\wd0 \kern.2pt \copy0 \kern-\wd0 \kern.2pt
      \copy0 \kern-\wd0 \kern.2pt \box0}

\def\sC{{\ss C}}

\def\sP{{\ss P}}

\def\sT{{\ss T}}

\def\xi{\tx{i}}

\newcommand{\ee}{\end{equation}}

\newcommand{\bea}{\begin{eqnarray}}
\newcommand{\eea}{\end{eqnarray}}
\newcommand{\beas}{\begin{eqnarray*}}
\newcommand{\eeas}{\end{eqnarray*}}

\newcommand{\R}{\mathbb{R}}

\newcommand{\nn}{\nonumber}

\mathchardef\za="710B  %\alpha
\mathchardef\zb="710C  %\beta
\mathchardef\zg="710D  %\gamma
\mathchardef\zd="710E  %\delta
\mathchardef\zve="710F %\epsilon
\mathchardef\zz="7110  %\zeta
\mathchardef\zh="7111  %\eta
\mathchardef\zvy="7112 %\theta
\mathchardef\zi="7113  %\iota
\mathchardef\zk="7114  %\kappa
\mathchardef\zl="7115  %\lambda
\mathchardef\zm="7116  %\mu
\mathchardef\zn="7117  %\nu
\mathchardef\zx="7118  %\xi
\mathchardef\zp="7119  %\pi
\mathchardef\zr="711A  %\rho
\mathchardef\zs="711B  %\sigma
\mathchardef\zt="711C  %\tau
\mathchardef\zu="711D  %\upsilon
\mathchardef\zvf="711E %\phi
\mathchardef\zq="711F  %\chi
\mathchardef\zc="7120  %\psi
\mathchardef\zw="7121  %\omega
\mathchardef\ze="7122  %\varepsilon
\mathchardef\zy="7123  %\vartheta
\mathchardef\zf="7124  %\varomega
\mathchardef\zvr="7125 %\varrho
\mathchardef\zvs="7126 %\varsigma
\mathchardef\zf="7127  %\varphi
\mathchardef\zG="7000  %\Gamma
\mathchardef\zD="7001  %\Delta
\mathchardef\zY="7002  %\Theta
\mathchardef\zL="7003  %\Lambda
\mathchardef\zX="7004  %\Xi
\mathchardef\zP="7005  %\Pi
\mathchardef\zS="7006  %\Sigma
\mathchardef\zU="7007  %\Upsilon
\mathchardef\zF="7008  %\Phi
\mathchardef\zW="700A  %\Omega

\newcommand{\epf}{\hfill$\Box$}
\newcommand{\bepf}{\noindent\textit{Proof.-} }

\def\be#1{\begin{equation}\label{#1}}
\newtheorem{theo}{Theorem}[section]
\newtheorem{prop}{Proposition}[section]
\newtheorem{lem}{Lemma}[section]

\newtheorem{definition}{Definition}[section]
\begin{document}
\title{Double Affine Bundles\thanks{Research financed by the Polish
Ministry of Science and Higher Education %, by means of the budget for science 2006-2009,
under the grant No. N201 005 31/0115.}}

        \author{
        Janusz Grabowski$^1$, Pawe\l\ Urba\'nski$^2$  Miko\l aj Rotkiewicz$^3$\\
        \\
         $^1$   {\it Institute of Mathematics}\\
                {\it Polish Academy of Sciences}\\
 $^2$ {\it Division of Mathematical Methods in Physics}\\
                {\it University of Warsaw} \\
                $^3$ {\it Institute of Mathematics}\\
                {\it University of Warsaw}
                }
\date{}
\maketitle
\begin{abstract} A theory of double affine and special double affine bundles, i.e. differential manifolds
with two compatible (special) affine bundle structures, is developed as an affine counterpart of the theory of
double vector bundles. The motivation and basic examples come from Analytical Mechanics, where double affine
bundles have been recognized as a proper geometrical tool in a  frame independent description of many
important systems. Different approaches to the (special) double affine bundles are compared and carefully
studied together with the problems of double vector bundle models and hulls, duality, and relations to
associated phase spaces, contact structures, and other canonical constructions.


\end{abstract}

\section{Introduction}
Double structures appear in geometry in a natural way, as the result of iteration of some functors. For
example, iterations of the tangent and cotangent functors result in  $\ctgT \tgT M$, $\tgT \tgT M$,
$\ctgT\ctgT M$, and a fundamental role played by these objects in Analytical Mechanics is well known. These
are canonical examples of double vector bundles \cite{Pr}, i.e. manifolds with two compatible vector bundle
structures. The Lagrangian mechanics of a system with the configuration manifold $M$ is based on a canonical
isomorphism between the double vector bundles  $\ctgT \tgT M$  and $\tgT \ctgT M$. On the other hand, the
Hamiltonian formulation of a dynamics is based on an isomorphism between $\ctgT \ctgT M$ and $\tgT \ctgT M$
and the Legendre transformation makes use of the canonical isomorphism of $\ctgT E$ and $\ctgT E^\star$ for
$E$ being a vector bundle (\cite{Tu}). Natural generalizations of these isomorphisms lead to a very useful
understanding of  a Lie algebroid structure as a certain morphism of double vector bundle structures. There
are many other instances in geometry, where double vector bundles emerge in a natural way (\cite{KU, Mac}).

On the other hand, we encounter in Physics many situations, where we are forced to replace vector-like objects
by affine ones, in order to obtain frame independent (gauge independent) formulations of our theories. For
example, Newtonian mechanics, relativistic mechanics of a charged particle and nonrelativistic mechanics of
nonautonomous systems require affine-like objects. Lagrangians (Hamiltonians) are no longer functions on
tangent (cotangent) bundles, but sections of affine bundles (\cite{TU,U, GGU1, GGU2}). Geometrical objects
which are suitable for these situations are provided by the geometry of affine values (AV-geometry) (\cite{U,
GGU1, GGU2}). Like in the case of vector bundles, we have to work with iteration of functors which give
objects with two compatible affine bundle structure.

The aim of this work is to develop a consistent theory of double objects in the category of affine bundles,
which generalize canonical objects known from the AV-geometry. The basic example (\cite{GGU1}) is the affine
phase bundle used in relativistic mechanics of a charged particle (for details see Section 5). We concentrate
on purely mathematical problems, as examples of applications being the starting point of our investigations
one can find in the papers cited above. We present a systematic introduction to double affine objects in both:
double affine bundles and special double affine bundles settings proving several theorems describing mutual
relations of the introduced objects. Our approach to double affine structures corresponds to the novel
approach to double vector bundles presented in (\cite{GR}).

The paper is organized as follows:

In Section~2 we  give two equivalent definitions of a {\it double affine bundle}. One uses compatibility
conditions for two affine bundle structures, while the other is given in terms of local trivializations. In
Section~3 we introduce the  notion of a {\it model double vector bundle} and a {\it vector hull} of a double
affine bundle. From the point of view of applications, the most important is the notion of a {\it special
affine bundle} and consequently, a {\it special double affine bundle}. We should view a special affine bundle
as a generalization of the product $E\times\R$,where $E$ is a vector bundle, with functions on $E$ interpreted
as sections of the trivial bundle $E\times\R \rightarrow E$.  Like vector bundles, special affine bundles form
a category with duality (which is no longer true for just affine bundles). We get also duality theorems,
similar to the ones known in the category of double vector bundles (\cite{Mac,KU}).  In Section~5 we analyze
the case of  double bundles which appear as a result of an application of the {\it phase functor} to a special
affine bundle. This generalizes the well-known structures of the cotangent of a vector bundle $\ctgT E$. The
canonical example, the {\it contact bundle} $\Contact A $ of a special double affine bundle, is given in
Section~6. It is obtained by applying  the contact functor to a special affine bundle $A$. There is a
canonical isomorphism of $\Contact A $ and $\Contact \affdual{A}$, where $\affdual{A}$ is the special affine
dual of $A$, which can be interpreted as a functorial version of the {\it universal Legendre transformation}
of Analytical Mechanics. This section contains also an interesting description of affine duality as being
encoded in a single geometrical object - the {\it double affine dual} $\mathbb B A$ - derived from the
cotangent bundle $\ctgT A$ and its canonical symplectic structure. The last section is devoted to natural
generalizations of all these concepts to $n$-affine bundles.

\section{Basic examples and definitions}
We assume that the reader has a basic knowledge about affine spaces, which can be found in many books on
linear algebra and linear geometry.

An affine space is usually defined as a triple $(A, V, +)$, where $+: A\times V\to A$ is a free and transitive
action of a vector space $V$ on a manifold $A$. Within this definition the particular vector space $V$ is
fixed and the affine combinations map $\aff_+ = \aff: A\times A\times \Real \to A$,
$$
\aff(a, b;\zl) :=   b + \zl \cdot [b, a]_V,
$$
where $[b, a]_V\in V$ is the unique vector such that $b + [b, a]_V = a$, does not determine the affine
structure. Indeed, for any linear automorphism $\zvf : V\to V$, the action $+': A\times V\to A$ defined by $a
+' v =  a + \zvf(v)$ gives the same  map $\aff_{+'} = \aff_+$. We prefer to  work with objects whose structure
is completely encoded in the affine combinations map. In order to do this, we formally define an {\it affine
space} as an abstract class of a triple $(A, V, +)$ as above subject to the following equivalence relation:
$(A, V, +)\sim (A', V', +')$ if and only if $A=A'$ and $\aff_+  =\aff_{+'}$. Let $(A,V, +)$ be a representant
of an affine space and consider the following relation on $A\times A$:
\be{e:rVA}
(a,b) \sim (a', b') \iff \aff(a, b'; 1/2)=\aff(b, a';1/2),
\ee
It follows immediately that it is an equivalence relation and its equivalence classes, denoted  by $[a,b]$,
$a, b\in A$, form a vector space, denoted by $\Vmodel(A)$, which is isomorphic to $V$. Within this
isomorphism, the action of $\Vmodel(A)$ reads as
\be{e:VA-action}
a+[a,b] = b.
\ee
The structure maps of $\Vmodel(A)$ can be expressed in terms of the map $\aff$,
$$
\zl\cdot[a,b] := [a, \aff(b, a; \zl)], \quad [a,b] + [a,c] = 2\cdot [a, \aff(b,c;1/2)],
$$
hence the vector space $\Vmodel(A)$ is defined independently on the choice of a representant $(A, V, +)$. This
shows that the considered affine space has a canonical representant $(A, \Vmodel(A), +)$. It allows us to
write simply $(A, \aff)$ instead of  the class of  $(A, V, +)$. The vector space $\Vmodel(A)$ is called the
{\it model vector space} of $(A, \aff)$.

We begin with a definition of a double affine bundle formulated in terms of natural maps for affine bundles.
Then we shall prove that this notion has also a nice description in local coordinate systems as in
(\cite{GGU1}).

\begin{definition}
A {\it double affine bundle} $\Abf=(A; A_1, A_2; M)$ is a commuting diagram of four affine bundles
\begin{equation}\label{4affbndl}
\xymatrix{ A\ar[r]^{\zp_2} \ar[d]^{\zp_1} & A_2\ar[d]^{\zp_1'} \\
A_1\ar[r]^{\zp_2'} & M }
\end{equation}
such that
\begin{enumerate}
\item[(i)] $(\pi_i, \pi_i')$, $i=1, 2$, are morphisms of affine bundles, \item[(ii)] $(\pi_1, \pi_2):A\to
A_1\times_M A_2$ is surjective.

For $x, y\in A$ being in the same fiber of $\pi_i$ and $\zl\in\Real$, let $\aff_i(x, y;\zl)\in A$ denote the
affine combination of $x$ and $y$ with weights $\zl$ and $1-\zl$, respectively. \item[(iii)] For each
$\zl\in\Real$, $i=1,2$,
\[
\aff_i(-, -; \zl): A\times_{\pi_i}A\to A
\]
is a morphism of affine bundles.
\end{enumerate}
\end{definition}

\noindent These conditions need a few explanation remarks.

\medskip
\noindent{\bf Remark 1}. The condition (ii) means that the morphisms $(\pi_i, \pi_i')$, $i=1, 2$, are
fiberwise surjective maps. Condition (i) does not imply (ii). For example, let us take $M =\{m\}$ to be a
single point, $A=\Real^2$ and let $\zp_1=\zp_2$ be the projection on the first factor $A_1=A_2=\Real$. Then
$\zp_1$ (and so $\zp_2$) is a morphism of affine bundles but it is constant on each fiber, so the map $(\zp_1,
\zp_2):A\to A_1\times_M A_2$ is not surjective. Let us mention that in a definition of a double vector bundle
it is enough to have a commuting diagram (\ref{4affbndl}) of four vector bundles and assume that homotheties
of $\zp_1$ and $\zp_2$ commute (see \cite{GR})). Then automatically the analog of $(ii)$  is satisfied.

\medskip
\noindent{\bf Remark 2}. Let us explain the condition (iii), say for $i=1$. Consider the horizontal inclusions
\[
\xymatrix{  \fbr{A}{A}{}\ar[d]^{\fbr{\pi_2}{\pi_2}{}} &  \fbr{A}{A}{\pi_1}\ar@{_(->} @<0.5ex>[l] \ar[d]^{\zt}
\ar@<-4ex>[r]^{\aff_1} & A\ar[d]^{\zp_2} \\
 \fbr{A_2}{A_2}{} & \fbr{A_2}{A_2}{\pi_1'}\ar@{_(->} @<0.5ex>[l] & A_2}
\]
where $\zt$ is the restriction of $\fbr{\pi_2}{\pi_2}{}$ to $\fbr{A}{A}{\pi_1}$. We shall show that
$\fbr{A}{A}{\zp_1}$ is an affine subbundle of $(\fbr{A}{A}{}, \fbr{\zp_2}{\zp_2}{})$. The condition (i)
implies that the fibers of $\zt$ are closed for taking affine combinations. Indeed, let $p=(p_1, p_2)$,
$q=(q_1, q_2)$ be in a fiber $F$ of $\zt$ and let $\zl\in \Real$. We want to show that $\aff_2(p, q; \zl)$
lies in $F$, i.e.
\[
\zp_1(\aff_2(p_1, q_1; \zl)) = \pi_1(\aff_2(p_2, q_2; \zl)).
\]
We use the fact that $\zp_1$ is an affine morphism and get
\bea\label{e:affsubndl}
\pi_1(\aff_2(p_1, q_1; \zl))&=&\aff_2(\pi_1(p_1), \pi_1(q_1); \zl)\nn \\
&=& \aff_2(\pi_1(p_2), \pi_1(q_2); \zl)  \\
&=& \pi_1(\aff_2(p_2, q_2; \zl)).\nn
\eea
First we prove that the canonical map $\zp := (\zp_1, \zp_2): A\to\fbr{A_1}{A_2}{M}$ is an affine bundle
projection with respect to both affine structures on $A$. Since $\zp_1$ is a morphism of affine bundles which
is surjective along the fibers (by condition (ii)), each fiber $F_{a_1, a_2} = (\zp_1, \zp_2)^{-1}(a_1, a_2)$,
$(a_1, a_2)\in\fbr{A_1}{A_2}{M}$, is an affine subspace of $\zp_2^{-1}(a_2)$ and the dimension of $F_{a_1,
a_2}$ does not depend on $(a_1,a_2)$. From symmetry, $F_{a_1, a_2}$ has also another affine space structure
inherited from $\aff_1$. We shall show later (Lemma \ref{l:twoaffstr}) that these two structures coincide.
Moreover, the fibers $F_{a_1, a_2}$ depend smoothly on $a_1$ and $a_2$, hence $\zp$ is a projection of a
locally trivial fibration, and because the fibers are affine spaces, it is an affine bundle projection. Hence
for any contractible open set $U\subset M$, $A_{|U} := (\zp_1'\circ\zp_2)^{-1}(U)$ admits a trivialization
$A_{|U} \simeq (\fbr{A_1}{A_2}{U})\times F$ for a fixed affine space $F$. Note that we do not claim here, that
$A$ is locally trivial as a double affine bundle, what we prove in the Theorem~\ref{eqas}. It follows now
easily that $\zt$ is also a locally trivial fibration, and because of (\ref{e:affsubndl}), it is an affine
subbundle, as we claimed. If (ii) is not satisfied then it may happen that some of the fibers of $\zt$ are
empty, as it is the case for the data from Remark 1.

\medskip
\noindent{\bf Remark 3}. Condition (iii) can be written in a form of interchange law
\be{intlaw}
\aff_2(\aff_1(x_1, x_2; \zl), \aff_1(y_1, y_2; \zl); \mu) = \aff_1(\aff_2(x_1, y_1; \mu), \aff_2(x_2, y_2;
\mu); \zl).
\ee

\noindent Note that the affine structure on $A_1$ (resp. $A_2$) is determined by affine combinations $\aff_2$
(resp. $\aff_1$) on $A$. This is so because $\zp_1$ (resp. $\zp_2$) is a morphism of affine bundles. Later on
we shall write simply $\aff_2$ (resp. $\aff_1$) for affine combinations on $A_1$ (resp. $\zp_2$).

\medskip
Another approach to double affine bundles is possible. In the paper (\cite{GGU1}) double affine bundles are
defined by means of gluing trivial double affine bundles

\begin{equation}
\xymatrix{U_\alpha\times K_1\times K_2\times K \ar[rr]^{pr_2} \ar[d]^{pr_1} && U_\alpha\times K_2\ar[d]  \\
U_\alpha\times K_1 \ar[rr] && U_\alpha }
\end{equation}
We glue such trivial double affine bundles by means of a family of isomorphisms of trivial double affine
bundles $\zvf_{\alpha, \beta}: (U_\alpha\cap U_\beta)\times K_1\times K_2\times K \circlearrowleft$. Let us
take affine coordinates $(y^j)$, $(z^a)$, $(c^u)$ on affine spaces $K_1$, $K_2$ and $K$, respectively. Then,
being an isomorphism of trivial double affine bundles means that the change of coordinates $\zvf_{\alpha,
\beta}: (x, y, z, c)\mapsto (x', {y}', {z}', {c}')$ has the form
\bea\label{coord}
x'&=&x'(x),\nn\\
{y^j}' &=& \alpha_0^j(x)+ \sum_i\alpha_i^j(x)y^i, \nn \\
{z^a}' &=& \beta_0^a(x) + \sum_b\beta_b^a(x)z^b, \\
{c^u}' &=& \gamma_{00}^u(x) + \sum_i \gamma_{i0}^u(x)y^i + \sum_b\gamma_{0b}^u(x)z^b +
\sum_{i,b}\gamma_{ib}^u(x)y^iz^b + \sum_w\sigma_w^u(x)c^w.\nn
\eea
We shall prove that these  two approaches are equivalent. Without loss of generality we assume throughout this
paper that the manifold $M$ is connected.
\begin{theo}\label{eqas} Let $\Abf=(A; A_1, A_2; M)$ be a double affine bundle as in (\ref{4affbndl}).
Then, there exist an open covering $\{U_\alpha\}_{\za\in\mathcal{I}}$ of $M$, affine spaces $K_1$, $K_2$, $K$
and diffeomorphisms
$$
\zvf_{\alpha}: A_{|U_\alpha} \to U_\alpha\times K_1\times K_2\times K
$$
inducing  isomorphisms of $A_{|U_\alpha}$ with trivial double affine bundle, such that the gluing
$\zvf_{\beta}\circ \zvf_{\alpha}^{-1}$ over $U_\alpha\cap U_\beta$ has the form (\ref{coord}).
\end{theo}
\noindent We have already observed that for each $(a_1, a_2) \in \fbr{A_1}{A_2}{M}$, the intersection of the
fibers $\pi_1^{-1}(\{a_1\})\cap\pi_2^{-1}(\{a_2\}) \subset A$ carries two structures of an affine space. From
the lemma below it follows that these structures coincide.
%**************************

\begin{lem}\label{l:twoaffstr}
If a manifold $A$ has two structures of an affine space, determined by affine combinations $\aff_1$ and
$\aff_2$ satisfying the interchange law (\ref{intlaw}), then $\aff_1 = \aff_2$.
\end{lem}
\bepf
Let us fix a point $\theta \in A$ and carry the structure of the model vector space of $(A, \aff_i)$, denoted
by $\Vmodel_i(A)$, $i=1,2$, to $A$ using isomorphisms
$$
 I_{\zvy, i} :  A\to \Vmodel_i(A),\quad a\mapsto [\zvy, a]_i,
$$
where $[a, b]_i$ denotes the vector from $a\in A$ to $b\in A$ with respect to the  affine structure $\aff_i$
on $A$. We shall show that these vector space structures coincide. In view of (\cite{GR}, Proposition 3.1), it
is enough to verify that
$$
\zl\mlt{1}\mu\mlt{2} a= \mu\mlt{2}\zl\mlt{1}a,
$$
where $a\in A$, $\zl, \mu\in \Real$, and $\zl\mlt{i}a$ stands for the scalar multiplication with respect to
$i$th vector space structure on $A$, i.e. $\zl\mlt{i}a = \aff_i(a, \zvy;\zl)$. Short calculations show that
\bea
\mu\mlt{2}\zl\mlt{1}a &=& \aff_2(\aff_1(a,\zvy;\zl),\zvy;\mu)\nn \\
&=&\aff_2(\aff_1(a,\zvy;\zl), \aff_1(\zvy,\zvy;\zl);\mu) \nn \\
&=&\aff_1(\aff_2(a, \zvy;\mu), \zvy;\zl) = \zl\mlt{1}\mu\mlt{2}a. \nn
\eea

\epf

Now we are ready to prove  Theorem \ref{eqas}. We shall  show that, by making suitable choices of
`zero-sections', one can equip any double affine bundle with a compatible structure of a double vector bundle.
This can be seen as a construction of the {\it model} double vector bundle of $A$, denoted by $\dblVmodel(A)$.
A detailed and equivalent but shorter description of $\dblVmodel(A)$ we postpone to the next section.

For simplicity, let us work first with the case of $M=\{m\}$ being a single point. Let us fix $\zvy \in A$ and
denote $\zvy_i=\zp_i(\zvy)$, $i=1,2$, which will later play the role of zero in the corresponding vector
spaces.

Note that, thanks to the condition (ii), the morphism $\zp_2:A\to A_2$ is  surjective on each fiber of
$\zp_1:A\to A_1$. Hence   $\zp_2^{-1}(\{\zvy_2\})\to A_1$ is an affine  subbundle of $\zp_1:A\to A_1$.  We
claim that it is possible to find a section $\zs_1$ of this subbundle such that its image  $\zs_1(A_1)$ is an
affine subspace of $(\zp_2^{-1}(\{\zvy_2\}), \aff_2)$, i.e. $\zs_1$ is a morphism of affine bundles from $A_1$
to $(A, \aff_2)$. Indeed, consider the affine map
$$
\tilde{\zp}_1: \zp_2^{-1}(\{\zvy_2\}) \to A_1, \quad \tilde{\zp}_1={\zp_1}_{|\zp_2^{-1}(\{\zvy_2\})}
$$
between the fibers over $\zvy_2$ and $m$, respectively. Let $W\subset \Vmodel_2(\zp_2^{-1}(\{\zvy_2\}))$ be
any complementary vector subspace to the kernel of the linear part $\tilde{\zp}_1^v$ of $\tilde{\zp}_1$. Here
$\Vmodel_i(A)$ ($i=1,2$) stands for the model vector bundle of the affine bundle $(A, \aff_i)$. The
restriction of $\tilde{\zp}_1^v$ to $W$ is a linear isomorphism from $W$ to $\Vmodel_2(A_1)$, the model vector
bundle of $A_1\to M$. We define $\zs_1$ in an obvious way such that $\zs_1(\zvy_1)=\zvy$ and the image
$\zs_1(A_1)$ is equal to the affine subspace $\zvy +_2 W$ of $\zp_2^{-1}(\{\zvy_2\})$, where $+_i$, $i=1, 2$,
stands for the action of the model of $(A, \aff_i)$ on $A$. We shall later use $\zs_1$ to get a vector bundle
structure on $A\to A_1$.

Let us work now over an arbitrary manifold $M$ and let $U$ be a small neighbourhood  of a given point in $M$.
Without loss of generality we may assume that $M=U$. From a smooth section $\zvy:M\to A$ of the fibration
$\zp_1'\circ \zp_2:A\to M$ one get the sections $\zvy_i\in Sec_M(A_i)$, $\zvy_i = \zp_i\circ\zvy$, and find
analogously the sections $\zs_i\in Sec_{A_i}(A)$, ($i=1,2$), which are morphisms of the affine bundles, and
moreover $\zp_2(\zs_1(A_1))$ lies in the image of the "zero section" of $\zp_1':A_2\to M$, i.e.
$\zp_2(\zs_1(A_1))\subset \zvy_2(M)$. Similarly, $\zp_1(\zs_2(A_2))$ is a subset of the image of the "zero
section" of $\zp_2':A_1\to M$.

The choices of $\zs_1, \zs_2$ put the vector bundle structures on the affine bundles $\zp_i:A\to A_i$,
$i=1,2$. We shall check that these two structures give  a double vector bundle structure. From (\cite{GR},
Theorem 3.1) it suffices to check that
\be{dblint}
\zl\mlt{1}\zm\mlt{2} a= \zm\mlt{2}\zl\mlt{1}a.
\ee
for $\zl, \mu \in \Real$. We have
$$
\zm\mlt{i} a = \aff_i(a, \zs_i(a_i); \zm),
$$
where $a_i = \zp_i(a)$, hence
\bea
\zl\mlt{1}\zm\mlt{2}a &=& \aff_1(\aff_2(a, \zs_2(a_2); \zm), \zs_1(\aff_2(a_1, \zvy_1; \zm)); \zl) \nn \\
&=& \aff_1(\aff_2(a, \zs_2(a_2); \zm), \aff_2(\zs_1(a_1), \zvy; \zm)); \zl).
\eea

\noindent We used the fact that $\zs_1$ is an affine bundle morphism. We get a similar formula for
$\zm\mlt{2}\zl\mlt{1}a$ from which we get (\ref{dblint}) by the interchange law (\ref{intlaw}). This shows
that locally  any double affine bundle can be given a compatible double vector structure. From the local
decomposition theorem of double vector bundles (\cite{KU,GR}) we get local identifications $\zvf_\alpha:
A_{|U_\alpha} \to U_\alpha\times K_1\times K_2\times K$. Because $\zvf_{U_\beta}\circ \zvf_{U_\alpha}^{-1}$ is
a morphism of double affine bundles, it has the form as in (\ref{coord}).

\epf

\noindent We shall call a coordinate system $(x, y^j, z^a, c^u)$ induced from a local decomposition of $A$ as
in Theorem \ref{eqas} an {\it adapted coordinate system} for $A$.

Let $l\in\Sec_M(\vectdual{E})$ be a linear function on a vector bundle $\zp:E\to M$ such that $0\neq l_m\in
\vectdual{E}_m$ for any $m\in M$. Then any level set of $l$, $A = \{x\in E: l(x)=c\}$, $c\in \Real$, is an
affine bundle modelled on the kernel of $l$. We shall denote it by $\subaff{E}{l}$, when $c=1$, and by
$\subaff{E}{l=c}$, in general. This construction can be even partially reversed: any affine bundle $A$ is
canonically embedded into a vector bundle called the {\it vector hull} of $A$ (\cite{GM,GGU2}). One can ask
about a similar passage from double vector bundles to double affine bundles. Let us assume that $(D; D_1, D_2;
M)$ is a double vector bundle and let $l_1$ be a linear function on $D_1$. Then the pullback $\tilde{l}_1$ of
$l_1$ with respect to the projection $D\to D_1$ is a linear function on $D$ with respect to the vector bundle
structure on $D\to D_2$, because it is a composition of the morphism $\zp_1$ and $l_1$, if we treat a linear
function on the total space as a morphism to the trivial bundle $M\times \Real\to M$. We can  interpret
$\tilde{l}_1$ as an element of $\Sec_{D_2}(\dblvecdual{D}{D_2})$, where we denote the dual of $D$ as a vector
bundle over $D_2$ by $\dblvecdual{D}{D_2}$. In the graded approach to double vector bundles, as in
(\cite{GR}), one can  simply view $l_1$ and $l_2$ as functions of degrees $(0,1)$ and $(1,0)$, respectively,
on the total  space $D$. We have the following

\begin{theo}\label{thm:constr} Let $(D; D_1, D_2; M)$ be a double vector bundle and let $l_i$ be
fiberwise non-zero linear functions on $D_i$, $i=1,2$. Let $\tilde{l}_1\in\Sec_{D_2}(\dblvecdual{D}{D_2})$ be
the corresponding pullback of $l_1$, and similarly for $\tilde{l}_2$. Let  $c_1, c_2\in \Real$, $A = \{x\in D:
\tilde{l}_1(x) = c_1,\, \tilde{l}_2(x) = c_2 \}$ and $A_i=\subaff{D_i}{l_i=c_i}$. Then
$$
\xymatrix{ A\ar[r]^{\zp_2} \ar[d]^{\zp_1} & A_2\ar[d]^{\zp_1'}  \\
A_1\ar[r]^{\zp_2'} & M }
$$
is a double affine bundle, where $\zp_i$, $i=1,2$, are the restrictions of the projections of $D$ onto $D_i$.
\end{theo}
\bepf
Because a restriction of a linear map to an affine subspace is an affine map, the condition (i) in the
definition of a double affine bundle is satisfied. Similarly, the interchange law of (iii) holds, because it
is so for double vector bundles. The condition (ii) follows immediately from the definition of $A$.

\epf

\noindent If $A'$ is a submanifold of the total space $A$ of a double affine bundle $\Abf = (A; A_1, A_2; M)$
such that $\Abf' = (A'; \zp_1(A'), \zp_2(A'); \zp_1'(\zp_2(A')))$ is a double affine bundle with the structure
maps induced from $\Abf$, then $\Abf'$ is called a {\it double affine subbundle} of $\Abf$. In the situation
described by the theorem above, when $c_1=c_2=1$, we  shall refer to $\Abf = (A; A_1, A_2; M)$ as to a  double
affine subbundle of $\Dbf=(D; D_1, D_2; M)$ given in $\Dbf$ by means of linear functionals $l_1$, $l_2$.

If $\zvf$ is a surjective vector bundle morphism from $E_1\to M$ to $E_2\to M$, which is identity on the base
$M$, and if $s\in\Sec_M(E_2)$, then the preimage $\zvf^{-1}(s(M))$ is an affine subbundle of $E_1$. In the
case of double structures we have the following.

\begin{prop} Let $(D; D_1, D_2; M)$, $(D'; D_1', D_2'; M)$
be  double vector bundles and let $\zvf$ be a surjective morphism of double vector bundles from $D$ to $D'$
such that $\phi_{|M} = Id_{|M}$. Let $s$ be a section of the fibration $D'\to M$ and let us assume that the
core of $D'$ is trivial. Then $\zvf^{-1}(s(M))$ is a double affine subbundle of $D$.
\end{prop}
\epf

The assumption of the triviality of the core bundle in the above proposition is essential as shows the
following example.

\medskip
\noindent {\bf Example}. Let $D = M\times \Real\times\Real\times \Real$ be a trivial double vector bundle
constructed from three trivial bundles of rank $1$. Consider a double bundle morphism $\zvf$,
$$
\zvf (m; x; y; z) := (m;x y)
$$
from $D$ to the trivial  double vector bundle $M\times\Real$ with the core of rank $1$. Then for the section
$s=1_M$, the set $\zvf^{-1}(s(M)) = \{(m; x, y; z):xy=1\}$ is not a double affine subbundle of $D$. Moreover,
in the following example
$$
A  = \{(m; x;y;z)\in D: x+y+z=1\}
$$
is a double affine subbundle of $D$ which cannot be presented in a form $\zvf^{-1}(s(M))$ for any double
vector bundle morphism $\zvf$ from $D$.

\section{Canonical constructions}

Let us recall that any affine space $A$ is modelled on a vector space $\Vmodel(A)$ which consists of vectors
$[a,b]\in\Vmodel(A)$, $a, b\in A$, where by definition $[a, b] = [a', b']$ if and only if $\aff(a, b';
1/2)=\aff(b, a';1/2)$. The vector space $\Vmodel(A)$ acts on $A$ according the rule $a+[a,b] = b$. Moreover,
any affine space $A$ can be canonically immersed into a vector space $\vhull{A}$ called the {\it vector hull}
of $A$. This immersion satisfies a universal property for affine functions (\cite{GGU2,GM}). The hull of $A$
has the dimension greater by one than $A$ and can be defined as a factor vector space of a free vector space
with basis $\{x_a: a\in A\}$ modulo a subspace spanned by vectors of the form $x_{\aff(a, b; \zl)} - \zl x_a -
(1-\zl)x_b$, $\zl\in\Real$, $a, b\in A$ (\cite{GGU2}).  There is a unique linear function $s: \vhull{A}\to
\Real$ which assigns $1$ to each class of $x_a$, $a\in A$. The subset of $\vhull{A}$ defined by equations
$s=1$, and respectively $s=0$, can be canonically identified with  $A$ and $\Vmodel(A)$, respectively. These
notions are extended naturally (fiberwise) to the case of affine bundles.

A similar constructions can be performed with double affine bundles. For notation, if $\zp:A\to M$ is an
affine bundle then  $\Vmodel(\zp):\Vmodel(A)\to M$ stands for its model vector bundle. If $A$ has several
affine structures, we add a subscript to $\Vmodel$ to indicate which structure is considered. If $\zvf:A\to B$
is an affine bundle morphism then the linear part of $\zvf$ is denoted by $\zvf^v:\Vmodel(A)\to\Vmodel(B)$.
The induced morphism between the vector hulls is denoted by $\vhull{\zvf}:\vhull{A}\to\vhull{B}$. For an
affine function $y$ on the total space $A$ of $\zp$, $\vhull{y}$ and $\und{y}$ denote the extension of $y$ to
a linear function on $\vhull{A}$ and the linear part of $y$, respectively. If $(x; y^j)$ is an adapted local
coordinate system for $\zp:A\to M$ then $(x; \und{y}^j)$ is the induced coordinate systems for $\Vmodel(A)$.
For the rest of this section  $\Abf = (A; A_1, A_2; M)$  is a double affine bundle as in (\ref{4affbndl}).

\subsection{Model double vector bundle}
Let us consider the model vector bundles of the vertical affine bundles of (\ref{4affbndl}) and the induced
morphism $(\zp_2^v, \zp_2')$ between them:
$$
\xymatrix{ \Vmodel_1(A)\ar[r]^{\zp_2^v} \ar[d]^{\Vmodel_1(\zp_1)} & \Vmodel(A_2)\ar[d]^{\Vmodel(\zp_1')} \\
A_1\ar[r]^{\zp_2'} & M }
$$
It will turn out that $\zp_2^v: \Vmodel_1(A)\to\Vmodel(A_2)$ is still an affine bundle (canonically) and
$\Vmodel_1(\zp_1)$ is an affine bundle morphism. Let us consider its linear part $\Vmodel_1(\zp_1)^v$:
\be{dblmod}
\xymatrix{ \Vmodel_2\Vmodel_1(A)\ar[rr]^{\Vmodel_2(\zp_2^v)} \ar[d]^{\Vmodel_1(\zp_1)^v} && \Vmodel(A_2)\ar[d]^{\Vmodel(\zp_1')}  \\
\Vmodel(A_1)\ar[rr]^{\Vmodel(\zp_2')} && M }
\ee
We shall  show that $\dblVmodel(A):=\Vmodel_2\Vmodel_1(A)$ is a double vector bundle and call it the {\it
model double vector bundle} of $\Abf$. Moreover, an analogous construction, relying on taking first the model
vector bundles with respect to the horizontal and then vertical affine structures of (\ref{4affbndl}), gives
an isomorphic double vector bundle $\Vmodel_1\Vmodel_2(A)$. The proof will use a local coordinate description
of $\Abf$.

Let $(x; y, z; c)$ and $(x'; y', z'; c')$  be two adapted local coordinate systems for $(A; A_1, A_2; M)$
related as in (\ref{coord}). Then $(x; y, \underline{z}; \underline{c})$ and $(x'; y', \underline{z}';
\underline{c}')$ are the induced local coordinate systems for $\Vmodel_1(A)$ related by
\bea\label{va}
x'&=&x'(x),\nn\\
{y^j}' &=& \alpha_0^j(x)+ \sum_i\alpha_i^j(x)y^i, \nn \\
{\underline{z}^a}' &=&  \sum_b\beta_b^a(x)\underline{z}^b, \\
{\underline{c}^u}' &=&   \sum_b\left(\gamma_{0b}^u(x)+\sum_{i}\gamma_{ib}^u(x)y^i\right) \underline{z}^b +
 \sum_w\sigma_w^u(x)\underline{c}^w.\nn
\eea
From this description one easily recognizes that $\Vmodel_1(A)$ is an affine bundle with the base manifold
$\Vmodel(A_2)$ and the fiber coordinates $(\underline{c}, y)$. For the model vector bundle $\dblVmodel(A)$,
one finds that the induced  coordinate system $(x, \und{y}, \und{z}; \dund{c})$ transforms according to:
\bea\label{e:dblmodeltrans}
x'&=&x'(x),\nn\\
{\und{y}^{j'}} &=& \sum_i\alpha_i^j(x)\und{y}^i, \nn \\
{\und{z}^{a'}} &=& \sum_b\beta_b^a(x)\und{z}^b, \\
{\dund{c}^{u'}} &=& \sum_{i,b}\gamma_{ib}^u(x)\und{y}^i\und{z}^b + \sum_w\sigma_w^u(x)\dund{c}^w\,,\nn
\eea
where all the indices start from 1. We can do the same construction in the reversed order getting the same
coordinates and the same transformation rules which shows that indeed $\dblVmodel(A)$ is a double vector
bundle and $\Vmodel_2\Vmodel_1(A)\simeq \Vmodel_1\Vmodel_2(A)$.

One can give a more geometric description of the affine structure of the bundle $\zp_2^v:
\Vmodel_1(A)\to\Vmodel(A_2)$ as follows. Suppose we are given two vectors: $v=[a_1, b_1]_1$, $w=[a_2, b_2]_1$,
$v,w \in \Vmodel_1(A)$ lying in the same fiber of $\zp_2^v$. Because
$\zp_1^{-1}(\{u_1\})\cap\zp_2^{-1}(\{u_2\})\neq\emptyset$ for any $(u_1,u_2)\in\fbr{A_1}{A_2}{M}$, we can find
$a_2'$, $b_2'$ such that $\zp_2(a_1) = \zp_2(a_2')$, $\zp_1(a_2')=\zp_1(a_2)$ and $[a_2', b_2']_1 = [a_2,
b_2]_1$. Then necessarily $\zp_2(b_2) = \zp_2(b_2')$ and we can define
$$
\aff(v, w; \zl) := [\aff_2(a_1, a_2'; \zl), \aff_2(b_1, b_2'; \zl)]_1.
$$
It makes sense, since $\zp_1$ is an affine morphism and so the head and the tail of the above vector lie in
the same fiber $\zp_1$. One can check that this construction gives the same affine structure as the one
described in local coordinates, and so it does not depend on the particular choices of $a_i$'s and $b_i$'s.

\subsection{Vector hull of a double affine bundle}
Now we are going to describe a construction of the {\it vector hull} of a double affine bundle. First we take
the vector hulls of the affine bundles $\zp_1:A\to A_1$ and $\zp_1':A_2\to M$ and get the induced morphism
$(\vhull{\zp}_2, \zp_2')$  between them:
$$
\xymatrix{ \vhull{A}^1\ar[r]^{\vhull{\zp_2}} \ar[d]^{\zp_1} & \vhull{A_2}\ar[d]^{\zp_1'}  \\
A_1\ar[r]^{\zp_2'} & M }
$$
We still use the same letters for projections from the total space of a vector hull on its base. We shall show
later that $\vhull{\zp_2}:\vhull{A}^1\to\vhull{A_2}$ is an affine bundle and $(\zp_1, \zp_1')$ is an affine
bundle morphism. Now we take the vector hulls  with respect to the horizontal affine bundle structures  and
get the induced morphism $(\vhull{\zp_1}, \zp_1')$ :
$$
\xymatrix{ \dblvhull{A}\ar[r]^{\vhull{\zp_2}} \ar[d]^{\vhull{\zp_1}} & \vhull{A_2}\ar[d]^{\zp_1'}  \\
\vhull{A_1}\ar[r]^{\zp_2'} & M }
$$
We shall show that this is a double vector bundle and that the construction is symmetric i.e. we obtain a
canonically isomorphic object if we take the hull with respect to the second $\aff_2$ structure in the first
step, and then apply the hull with respect to $\aff_1$. It will be called the {\it hull} of a double affine
bundle $\Abf$ and denoted by $\dblvhull{\Abf}$.

Recall that if $A$ is  an affine space and $(y^j)$ is a system of affine coordinates on $A$, then
$(\vhull{y}^j, s)$ is the induced system of linear coordinates  on the vector hull $\vhull{A}$ of $A$, where
$s$ is the unique extension of  the constant function $1_A$ on $A$  to a linear function on $\vhull{A}$. If
$A$ is an affine bundle over $M$ and $(x; y^j)$ is an adapted local coordinate systems on $A$, which
transforms as
$$
y^{j'} = \alpha_0^j(x)+ \sum_i\alpha_i^j(x)y^i,
$$
then the corresponding coordinates on the vector bundle hull $\vhull{A}$ transform in the following way:
$$
{\vhull{y}^{j'}} =  \alpha_0^j(x)s+ \sum_i\alpha_i^j(x)\vhull{y}^i, \quad s'=s.
$$
Let us go now to  a double affine bundle setting. With local coordinates $(x; y^j, \vhull{z}^a; \vhull{c}^u;
s)$ on $\vhull{A}^1$ and then with $(x; \vhull{y}^j, \vhull{z}^a; \dblvhull{c}^u; s, t)$  on $\dblvhull{\Abf}$
one finds that the corresponding transformations are the linearizations  of  (\ref{coord}), that is
\bea\label{e:dblvhulltrans}
{\vhull{y}^{j'}} &=& \alpha_0^j(x)t+ \sum_i\alpha_i^j(x)\vhull{y}^i, \nn \\
{\vhull{z}^{a'}} &=& \beta_0^a(x)s + \sum_b\beta_b^a(x)\vhull{z}^b, \\
{\dblvhull{c}^{u'}} &=& \gamma_{00}^u(x)st + \sum_i \gamma_{i0}^u(x)\vhull{y}^is +
\sum_b\gamma_{0b}^u(x)\vhull{z}^bt + \sum_{i,b}\gamma_{ib}^u(x)\vhull{y}^i\vhull{z}^b +
\sum_w\sigma_w^u(x)\dblvhull{c}^w.\nn
\eea
One can get a more compact formula by putting $\vhull{y}^0:=t$, $\vhull{z}^0 =s$.

Note that the main objects associated with a double affine bundle $\Abf$ are canonically embedded into the
hull of $\Abf$. In  particular, we find that the cores of $\dblvhull{\Abf}$ and $\dblVmodel(\Abf)$ are
isomorphic.

\begin{prop}\label{prop: isohulldvb}  Let $(A, A_1, A_2; M)$ be given in a double vector bundle $(D; D_1, D_2; M)$ by means
of linear functions $l_1$, $l_2$ as in Theorem \ref{thm:constr} (with $c_1 c_2\neq 0$). Then the vector hull
of $A$ is canonically isomorphic to $D$. Moreover, the model double vector bundle of $A$ is a double vector
subbundle of $D$ given by equations $l_1=0=l_2$.
\end{prop}

\bepf Without loss of generality we may assume that $c_1=c_2=1$.
A similar fact is well known in the categories of vector spaces and vector bundles: if $l\in \vectdual{V}$,
$l\neq0$, is a linear function on a vector space $V$ then $A=\subaff{V}{l}$ is an affine subspace of $V$ and
its vector hull $\vhull{A}$ is naturally isomorphic to $V$. Let us apply it to the vector bundle
$\subaff{D}{\tilde{l}_1}\to A_1=\subaff{D_1}{l_1}$ and its affine subbundle given by the equation
$\tilde{l_2}=1$, $A= \subaff{(\subaff{D}{\tilde{l}_1})}{\tilde{l}_2} \to A_1$. It follows that  the vector
hull $\vhull{A}^1\to A_1$ is canonically isomorphic with the vector bundle $\subaff{D}{\tilde{l}_1}\to A_1$.
The total space $\subaff{D}{\tilde{l}_1}$ is also an affine subbundle of $D\to A_2$ as a subset of points
satisfying  $\tilde{l}_1 = 1$. Hence the vector hull of $\subaff{D}{\tilde{l}_1}\to A_2$ is $D\to A_2$.
Reassuming, we have performed, as in definition, the construction of the hull of the double affine bundle
$\Abf$, and eventually arrived at the double vector bundle $\Dbf$, what justifies the first assertion. The
second one is clear when we compare the transformations of (\ref{e:dblmodeltrans}) and (\ref{e:dblvhulltrans})
with $s=t=0$.

\epf

\section{Special double affine bundles and duality}

A {\it special affine space} $\spec{A}=(A, v_A)$ is an affine space $A$ with a distinguished non-zero element
$v_A\in\Vmodel(A)$ in the model vector space. A vector space with a distinguished non-zero element is also
called {\it special}. It is known (\cite{U, GGU, GGU1}) that special affine spaces, in contrast to ordinary
affine spaces, have well-defined dual objects in the same category. Let $\spec{I} =(\Real, 1)$ be the special
vector space $\Real$ with the distinguished element $1$. The dual of $\spec{A} = (A,v_A)$ is, by definition,
the space of morphisms from $\Abf$ to $\spec{I}$ and is denoted by $\affdual{A}$. It consists of all affine
maps from $A$ to $\Real$ whose linear part preserve the distinguished elements. We shall call them {\it
special} affine maps. The model vector space of $\affdual{A}$ is
$$
\Vmodel(\affdual{A}) \simeq \{\zvf:A\to \Real |\,\, \zvf\,\, \text{is an affine map}, \phi^v(v_A) = 0\},
$$
where $\phi^v: \Vmodel(A)\to \Real$ is the linear part of $\phi$. Hence $\affdual{\spec{A}}= (\affdual{A},
1_A)$ is a special affine space with the constant function $1_A\in \Vmodel(\affdual{A})$ as a distinguished
element. In finite dimensions we have a true duality: $\affdual{(\affdual{\spec{A}})} \simeq \spec{A}$ for a
special affine space $\spec{A}= (A, v_A)$. The above notions and statements can be automatically extended to
the case of affine bundles.

On the other hand, if we have a double vector bundle
\be{e:dblvbndl}
\xymatrix{ D\ar[r]^{\zp_2} \ar[d]^{\zp_1} & D_2\ar[d]^{\zp_1'} \\
D_1\ar[r]^{\zp_2'} & M }
\ee
with the core bundle $D_3\to M$, which we shortly denote by $\Dbf = (D; D_1, D_2; M)$, then the total space
$\dblvecdual{D}{D_1}$ of the dual bundle to $\zp_1:D\to D_1$ has a double vector bundle structure
$$
\xymatrix{\dblvecdual{D}{D_1} \ar[r] \ar[d] & {\vectdual{D_3}}\ar[d]\\
D_1\ar[r] & M }
$$
with the core bundle isomorphic to $\vectdual{D_2}$, which we call the vertical dual of $\Dbf$ (\cite{KU,Mac})
and denote by $\Dbf^V$. Similarly we can form a horizontal dual of $\Dbf$ which we denote by $\Dbf^H =
(\dblvecdual{D}{D_2}; \vectdual{D_3}, D_2; M)$.

We are going to join both concepts and define a category where dualities of double affine bundles live in.

Let us recall that if $A$ is an affine space then the space of affine maps $A^\dag = \Aff(A, \Real)$ is an
example of a special vector space with the constant function $1_A$ as a distinguished element. If $\Vbf = (V,
v)$ is a special vector space then $\ahyp{V} = \{\phi\in \vectdual{V}: \phi(v) =1\} \subset \vectdual{V}$ is
an affine subspace of codimension $1$.

For the rest of this section,  $\Abf = (A; A_1, A_2; M)$ is as in (\ref{4affbndl}), and $\Dbf = (D; D_1,
D_2;M)$ is the hull of $\Abf$ with the core denoted by $D_3$. We shall also call $D_3$ the {\it core of
$\Abf$}.

\begin{definition}
{\rm The objects of the {\it category of  special double affine bundles} are double affine bundles
(\ref{4affbndl}) equipped with a distinguished nowhere vanishing section $\zs\in Sec_M(D_3)$. Morphisms are
assumed to preserve the distinguished elements, i.e. the induced map between the cores is a morphism of
special vector bundles.}
\end{definition}

\begin{prop}\label{prop:inducedSections} A nowhere vanishing section $\zs\in Sec_M(D_3)$
induces sections $\zs_i\in Sec_{A_i}(\Vmodel_i(A))$, $i=1,2$, what turns $\zp_i:A\to A_i$ into a special
affine bundle.
\end{prop}

\bepf
If $(D; D_1, D_2; M)$ is a double vector bundle with the core $D_3$ then the pullback of the core bundle,
$\vectdual{(\zp_2')}(D_3)$, with respect to $\zp_2':D_1\to M$ is a vector subbundle of $\zp_1:D\to D_1$,
because it can be identified with the kernel of the morphism $\zp_2$ (\cite{Mac}). Given a section $\zs$ as in
the hypothesis, we get the section
$$\tilde{\zs}_1 := \vectdual{(\zp_2')}(\zs) \in\Sec_{D_1}(\vectdual{(\zp_2')}(D_3)) \subset \Sec_{D_1}(D)$$
and similarly $\tilde{\zs}_2 \in \Sec_{D_2}(D)$. Now assume that $(D; D_1, D_2; M)$ is the  hull of $\Abf$ and
define $\zs_1:={\tilde{\zs}}_{1|A_1}$. It is straightforward to check that the image $\zs_1(A_1)$ is a subset
of $\Vmodel_1(A)\subset D$, the total space of the model vector bundle of $\zp_1: A \to A_1$. We view $\zs_1$
as a special section for this  affine bundle.

\epf

\noindent The core bundle of the double vector bundle $(\tgT E; E, \tgT M; M)$, associated with a vector
bundle $E\to M$, is canonically isomorphic to $E$. Indeed, it consists of vertical vectors along the zero
section. The construction $\zs\mapsto\tilde{\zs}_1$ from the proposition above  coincides with the well-known
one of the {\it vertical lift}  $\Sec_M(E)\to \Sec_E(\tgT E)$.
\begin{theo}\label{thm:dual} Let $\Abf = (A; A_1, A_2; M)$ be a special double affine bundle and
$\Dbf_3 = (D_3, \zs)$ be its core. Then
$$
\xymatrix{ \dblaffdual{A}{A_1}\ar[r]^{} \ar[d]^{} & \ahyp{D_3}\ar[d]^{} \\
A_1\ar[r]^{} & M }
$$
is also a special double affine bundle with the core equal to ($A_2^\dag, 1_{A_2})$.
\end{theo}

\noindent There is even a simpler description of duals to a special double affine bundle. Let us assume that a
special double affine bundle $\Abf$ is given in $\Dbf$ by means of  linear functions $l_1$, $l_2$ on the
bundles $D_1$, $D_2$, respectively. Let $\zs\in\Sec_M(D_3)$ be the distinguished nowhere vanishing section for
$\Abf$. Let us consider $\zs$ as a linear function $l_3$ on $\vectdual{D_3}$. We shall prove the following
\begin{theo}\label{thm:dual2}
The double affine bundle determined by the linear functions $l_1$ and $l_3$ in the vertical dual double vector
bundle $(\dblvecdual{D}{D_1}; D_1, \vectdual{D_3}; M)$ is canonically isomorphic to the double affine bundle
from Theorem \ref{thm:dual}. Moreover, $l_2$ is the distinguished section of the core bundle
$\vectdual{D_2}\to M$ of $\dblvecdual{D}{D_1}$ which corresponds to $1_{A_2}\in\Sec_M(A_2^{\dagger})$.
\end{theo}
\bepf (of Theorems \ref{thm:dual} and \ref{thm:dual2})
We can assume that $\Abf$ is given in $\Dbf$ by means of linear functions $l_1$, $l_2$. We shall recognize
$\dblaffdual{A}{A_1}$ as the subset
\be{e:subsetA}
\dblaffdual{A}{A_1} = \{\zf\in \dblvecdual{D}{D_1}: \tilde{l}_1(\zf) = \tilde{l}_3(\zf) =1 \}
\ee
where $l_3$ is the linear function on $\vectdual{D_3}$ associated with the section $\zs\in\Sec_M(D_3)$ and
$\tilde{l}_3$ is the pullback of $l_3$ with respect to the bundle projection $\dblvecdual{D}{D_1}\to
\vectdual{D_3}$. We have $l_3(\zf_m) = \langle \zf_m, \zs(m)\rangle$ for $m\in M$ and $\zf_m\in
{(\vectdual{D_3})}_m$, hence
$$
D_3^{\ddagger} = \{\zf_m\in \vectdual{D_3} :\langle \zf_m, \zs(m)\rangle =1, m\in M \} =
\subaff{(\vectdual{D_3})}{l_3}.
$$
Let $\tilde{\zs}_1 \in\Sec_{D_1}(D)$, $\zs_1 = \tilde{\zs}_{1|A_1}$, be the distinguished  elements induced
from the section $\zs$ of the core bundle, as in the proof of Proposition \ref{prop:inducedSections}. For
$\zf\in \dblvecdual{D}{D_1}$, which lies in a fiber over $x\in D_1$, we have
$$
\langle \tilde{l_3}, \zf \rangle  = \langle \zf, \tilde{\zs}_1(x)\rangle,
$$
hence  $\zf\in \dblaffdual{A}{A_1} \subset \dblvecdual{D}{D_1}$ if and only if $ \langle \tilde{l_3}, \zf
\rangle = \langle \zf, \zs_1(x)\rangle =1 $ for $x\in A_1\subset D_1$, what proves (\ref{e:subsetA}). This
way, in the vertical dual bundle $\Dbf^V = (\dblvecdual{D}{D_1}; D_1, \vectdual{D_3}; M)$ we have
distinguished the affine subbundles
$$
\xymatrix { **[l] \dblaffdual{A}{A_1}\subset \dblvecdual{D}{D_1}\ar[r] \ar[d]^{} & **[r] {\vectdual{D_3}
\supset D_3^{\ddagger}}\ar[d]^{} \\
**[l] {A_1\subset D_1} \ar[r]^{} & M }
$$
Moreover, the core of $\dblaffdual{A}{A_1}$ is  the core of $\Dbf^V$ and is equal to $\vectdual{D_2}\simeq
\Aff(A_2, \Real) =  A_2^{\dagger}$. It is equipped with the section $l_2 \in \Sec_M(\vectdual{D_2})$, which
corresponds to $1_{A_2}$, as $l_{2|A_2} = 1_{A_2}$.

\epf

\noindent Recall that, if $(A, v_A)$ is a special affine bundle, then its {\it adjoint} is the special affine
bundle $(A, -v_A)$. If $(\Abf, \zs)$ is a special double affine bundle,where $\zs$ is the distinguished
section of the core bundle,  then its adjoint is $(\Abf, -\zs)$. We shall denote with $\Dbf^f  = (D; D_2, D_1;
M)$ and $\Abf^f = (A; A_2, A_1; M)$ the flips of $\Dbf$ and $\Abf$, respectively. We consider the flip as an
operation on double vector (or affine) bundles relying on switching  the horizontal and vertical arrows in the
diagrams representing them. If $\Abf$ is special then so is $\Abf^f$ with the same distinguished section of
the core bundle. Let $\zp_D:D\to D_1\times_M D_2$ denotes the canonical fibration.

\begin{prop}\label{prop:affspduality} With the assumptions of Theorem \ref{thm:dual}, the special affine bundles
$\dblaffdual{\Abf}{A_1}\to \ahyp{D}_3$ and the adjoint of $\dblaffdual{\Abf}{A_2}\to \ahyp{D}_3$ are in
duality. The special double affine bundle $\Abf^{HVH}$ is naturally isomorphic to the adjoint of $\Abf^f$.
\end{prop}
\bepf
Let $\Dbf$ be as usual, the hull of $\Abf$. For $\Phi \in \dblvecdual{D}{D_1}$ and $\Psi \in
\dblvecdual{D}{D_2}$ such that $\zp_{D^V}(\Phi) = (d_1, \zf) \in \fbr{D_1}{\vectdual{D_3}}{M}$ and
$\zp_{D^H}(\Psi) = (\zf, d_2) \in\fbr{\vectdual{D_3}}{D_2}{M}$ the difference
$$
\langle \Phi, \Psi \rangle = \Phi(x) - \Psi(x),
$$
where $x\in D$ is any element lying over $(d_1, d_2)\in\fbr{D_1}{D_2}{M}$, does not depend on the choice of
$x$ and defines a non-degenerate pairing between $\dblvecdual{D}{D_1}$ and $\dblvecdual{D}{D_2}$ (\cite{KU}).
We shall restrict the pairing $\langle\cdot, \cdot\rangle$ to the affine subbundles $\dblaffdual{A}{A_1}$ and
$\dblaffdual{A}{A_2}$. The sections $l_2$, $l_1$ of the core bundles $\vectdual{D_2}$, $\vectdual{D_1}$ of the
double vector bundles $\dblvecdual{\Dbf}{D_1}$, $\dblvecdual{\Dbf}{D_2}$, respectively, induce the sections
$\bar{l}_2\in \Sec_{\vectdual{D_3}}(\dblvecdual{D}{D_1})$, $\bar{l}_1\in
\Sec_{\vectdual{D_3}}(\dblvecdual{D}{D_2})$. After the restriction to $\ahyp{D_3}\subset \vectdual{D_3}$ we
get the distinguished  sections of the model bundles of $\dblaffdual{A}{A_j} \to \ahyp{D_3}$, $j=1, 2$, which
we still denote by $\bar{l}_2$, $\bar{l}_1$. The sum  $\Phi + \bar{l}_2 (\zf)$ of two elements in the same
fiber of $\dblvecdual{D}{D_1}\to\vectdual{D_3}$, evaluated on $x\in D$, gives $\Phi(x) + l_2(d_2)$. Since
$\bar{l}_{j|A_j} = 1_{A_j}$, and $d_j\in A_j$, for $j=1, 2$, we have
$$
\langle\Phi + \bar{l}_2(\zf), \Psi\rangle = \langle \Phi, \Psi\rangle +1 = \langle \Phi, \Psi -
\bar{l}_1(\zf)\rangle,
$$
what justifies the first statement of our proposition.

If we encode the special double affine bundle $A$ as in the theorem in the form $(\Dbf; l_1, l_2; l_3)$ (the
special double affine subbundle of $\Dbf$ determined by the linear functions $l_1$, $l_2$ on the bases of the
side bundles and a linear function $l_3$ on the dual to the core bundle), then the horizontal dual $\Abf^H =
\dblaffdual{\Abf}{A_2}$ goes with $(\Dbf^H; l_3, l_2; l_1)$. Consequently, $\Abf^{HV}$ corresponds to
$(\Dbf^{HV}; l_3, l_1; l_2)$ and finally $\Abf^{HVH}$ to $(\Dbf^{HVH}; l_2, l_1; l_3)$. Moreover, it is known
that the double vector bundle $\Dbf^{HVH}$ obtained from $\Dbf$ by first taking horizontal dual, then
vertical, and again horizontal dual, is naturally isomorphic to the flip $\Dbf^f$ of $\Dbf$ (\cite{KU,Mac}).
The isomorphism $\za: \Dbf^{HVH}\to\Dbf^f$ is the identity on the side bundles $D_2$ and $D_1$, but is minus
the identity on the core $D_3$. Let us restrict $\za$ to the double affine subbundle $\Abf^{HVH}$. The
distinguished section $\zs$ (corresponding to $l_3$ above)  of the core bundle is moved to $-\zs$, hence $\za$
induces an isomorphism between $\Abf^{HVH}$ and the adjoint of $\Abf^f$.

\epf

\section{The affine phase bundle}

Let $\zz:Z\to M$ be a bundle of affine values (AV-bundle, in short). By definition (\cite{GGU2}), it is a
special  affine bundle of rank $1$. It follows that the model vector bundle of $\zz$ is the trivial vector
bundle $M\times \Real \to M$. The action of the model bundle on the total space $Z$ allows us to consider
$\zz$ as $(\Real, +)$-principal bundle with the base $M$. The {\it affine phase bundle} (\cite{U,GGU2})
$\Phase\zz : \Phase Z\to M$ is defined as follows. Let us consider the following equivalence relation on the
set of pairs $(m, \zs)$, $m\in M$, $\zs\in\Sec(\zz)$:
$$
(m, \zs)\sim (m', \zs') \quad\text{if and  only if}\quad m=m'\quad\text{and}\quad d_m(\zs-\zs') = 0.
$$
Here we identify $\zs-\zs'$ with a section of the model bundle $M\times\Real\to M$, i.e. a function on $M$.
The equivalence classes of $\sim$ are the elements of $\Phase Z$. It is an affine bundle modelled on $\ctgT
M$. The class of $(m, \zs)$ is denoted by $\affdiff_m\zs$ and called an {\it affine differential} of $\zs$ at
$m$.

Let $\zh:A\to M$ now be a special affine bundle with the distinguished nowhere vanishing section
$v_A\in\Sec_M(\Vmodel(A))$ and let $\zz:A\to\und{A}$, $\und{A}=A/\langle v_A\rangle$, be the associated
AV-bundle. We shall use the convention of (\cite{GGU2}), that the distinguished section of $\zz$ is $-v_A(m)$
at points $\und{a}\in\und{A}$ lying over $m\in M$. The total space of the affine phase bundle  $\Phase\zz:
\Phase A\to \und{A}$, has also another structure of an affine bundle, over the base $\und{\affdual{A}} =
\affdual{A}/\langle 1_A \rangle $, what makes $\Phase A$ a canonical example of a double affine bundle
depicted in the diagram
\begin{equation}\label{e:PA}
\xymatrix{ \Phase A\ar[r]^{\affdual{\Phase}\zeta} \ar[d]^{\Phase\zeta} &
\underline{\affdual{A}}\ar[d]^{\underline{\affdual{\eta}}}  \\
\underline{A}\ar[r]^{\underline{\eta}} & M }.
\end{equation}
In this section, basing on the example of  $\Phase A$, we are going to describe canonical objects associated
with double affine bundles. We shall give a clear picture  of the  hull, the  model double vector bundles  and
the duals of $\Phase A$.

We are going first to describe the double affine structure on $\Phase A$. Let us start with the projection in
the affine bundle $\affdual{\Phase}\zz:\Phase A\to\und{\affdual{A}}$. Let $\zw_{\und{a}}$ be any element (an
affine covector) of $\Phase A$, $a\in A$, $\zh(a) = m\in M$, $\und{a} = \zz(a)$, $(\Phase\zz)(\zw_{\und{a}}) =
\und{a}$. Let us write it in a form  $\zw_{\und{a}} = \affdiff_{\und{a}}\zs$, for a section
$\zs\in\Sec_{\und{A}}(A)$. It is possible to find an {\it affine} section representing $\zw_{\und{a}}$, i.e.
we may assume that $\zs$ is an affine morphism from $\und{\zh}:\und{A} \to M$ to $\zh:A\to M$. The space of
affine sections of $\zz:A\to\und{A}$ is in $1$-$1$ correspondence $\zs\mapsto f_{\zs}$ with the set of special
affine maps on $A$, and hence with $\Sec_M(\affdual{A})$ (\cite{U}), where  $f_\zs : A\to\Real$ is defined by
\be{e:ProjOnDualA}
f_\zs(x) \cdot v_A(m') = [\zs(\und{x}), x] \in \Vmodel(A_{m'}),
\ee
for $x\in A$, $m'=\zh(x)$. We have $f_\zs(x + v_A(m'))= f_\zs(x)+1$, so $f_\zs$ is a special affine map.
Moreover,  if $\zw_{\und{a}} = \affdiff_{\und{a}}\zs'$ for an affine section $\zs'$, then
$d_{\und{a}}(\zs-\zs') = 0$, hence $f_\zs - f_{\zs'}$ is constant on $A_m$. Thus we get a well-defined map
$\affdual{\Phase}\zz: \Phase A\to\und{\affdual{A}}$, $\zw_{\und{a}}\mapsto f_{\zs|A_m} + \Real \cdot 1_{A_m}
\in \und{\affdual{A_m}}$. We are going now to define the affine structure in a fiber of $\affdual{\Phase}\zz$.
Let us fix $m\in M$, $\und{a_1}$, $\und{a_2} \in \und{A_m}$, $f_0\in\affdual{A_m}$ and consider the fiber $F =
(\affdual{\Phase}\zz)^{-1}(\und{f_0})$ and the set $\sum$ of all affine sections $\zs \in \Sec_{\und{A}}(A)$
such that $\affdiff_{\und{a}}\zs \in F$ at some point $\und{a}\in \und{A_m}$. The set $\sum$ does not depend
on the choice of $\und{a}$ from $\und{A_m}$. Let us choose a section $\zs_0\in\sum$. Then any element $\zw \in
F$ is of the form
$$
\zw = \affdiff_{\und{a}}\zs_0 + \za
$$
for an $\und{a}\in \und{A_m}$ and a unique $\za\in \ctgT _m M$, where $\ctgT _m M$ is considered as a subspace
of the model space $\ctgT_{\und{a}} \und{A} \simeq \Vmodel(\Phase_{\und{a}}A)$ via the pullback map with
respect to the projection $\und{A}\to M$. Indeed, the difference of two affine forms from $F$, $\zw -
\affdiff_{\und{a}}\zs_0 \in \ctgT_{\und{a}} \und{A}$, is  zero on any vector tangent to the fiber
$\und{\zh}^{-1}(m)$ and hence it is the pullback of a 1-form on $M$. Let us define the affine combinations map
in $F$ by the formula
\be{e:affstrPA}
\aff(\affdiff_{\und{a}_1}\zs_0 + \za_1, \affdiff_{\und{a}_2}\zs_0 + \za_2; \zl) = \affdiff_{\und{a}}\zs_0 +
(\zl (\za_1 + (1-\zl) \za_2),
\ee
where $\und{a}=\aff(\und{a}_1, \und{a}_2; \zl)$ and $\za_i \in \ctgT_mM$, $i=1, 2$. It follows easily from the
above discussion that this definition does not depend on the choice of $\zs_0$. Indeed, if $\zs, \zs' \in\sum$
are such that $\affdiff_{\und{a_i}}\zs = \affdiff_{\und{a_i}}\zs'$ for $i=1, 2$, then $\affdiff_{\und{a}}\zs =
\affdiff_{\und{a}}\zs'$ for any affine combination $\und{a} = \aff(\und{a_1}, \und{a_2}; \zl)$, because the
function $\zs-\zs' : \und{A}\to \Real$ is affine.

A vector space $V$ with two distinguished non-zero elements $v_0\in V$ and $\za_0\in \vectdual{V}$ such that
$\za_0(v_0)=0$ is called a {\it bispecial} vector space (\cite{GGU2}).  If $(A, v_A)$ is a special affine
space then its vector hull $\vhull{A}$ is canonically a bispecial vector space with distinguished elements
$v_A\in \Vmodel(A)\subset\vhull{A}$ and the unique function $\za_A\in \vectdual{(\vhull{A})}$ for which $A$ in
$\vhull{A}$ is defined by the equation $\za_A = 1$. Now we assume that $(\zh, v_A)$, $\zh: A\to M$, is a
special affine bundle. Then $\vhull{A}$ is a bispecial vector bundle over $M$. We shall use the letter $E$ for
$\vhull{A}$. Let us analyze the cotangent bundle $\ctgT E$. The action $\zvf$ of $(\Real, +)$, $\zvf_t :
v_m\mapsto v_m + t\cdot v_A(m)$ on $E$, can be lifted to an $(\Real, +)$-action, say $\zc_1$, on $\ctgT E$ by
means of pullback
\be{e:action1}
(\zc_1)_t := (\zvf_{-t})^*.
\ee
We have also another $(\Real, +)$-action on $\ctgT E$ given by
\be{e:action}
\zc_2(t, \zw_v) = \zw_v + t\cdot d_v\za_A
\ee
for $v\in\vhull{A}$, $\zw_v\in \ctgT_v E$, $t\in \Real$. Obviously the actions $\zc_1$ and $\zc_2$ commute and
give rise to an action $\zc = (\zc_1, \zc_2)$ of $\Real\times \Real$ on $\ctgT E$. We shall show that the
orbits of $\zc$ form the vector hull of $\Phase A$.

\begin{theo}\label{thm:hullPA} Let $\zh:(A, v_A) \to M$ be a special affine bundle and let $\zc$ be the canonical action
of $\Real\times \Real$ on ${\ctgT} E$ defined above. The orbit space of $\zc$, denoted by $\affctg{A} = \ctgT
E/\Real\times\Real$, has a well-defined structure of a double vector bundle
\be{e:SdA}
\xymatrix{\affctg{A} \ar[r] \ar[d] & \und{\vectdual{E}}\ar[d] \\
\und{E}\ar[r] & M }
\ee
inherited from $(\ctgT E; E, \vectdual{E}; M)$. It is canonically isomorphic to the double vector hull of the
double affine phase  bundle $\Phase A$ depicted in (\ref{e:PA}).
\end{theo}
\bepf
By definition, $\und{E} = E/\langle v_A\rangle$ and $\und{\vectdual{E}} = \vectdual{E}/\langle \za_A\rangle$.
Let us remark the obvious isomorphisms: $\vhull{(\und{A})} \simeq \und{E}$ and $\vhull{\affdual{A}}\simeq
\Aff(A, \Real) \simeq \vectdual{E}$. Note however, that there is no canonical isomorphism between the dual of
$\und{E}$ and $\und{\vectdual{E}}$ and so the double vector bundle in the theorem  is not of the type $(\ctgT
F; F, \vectdual{F}; M)$ for a vector bundle $F\to M$. Indeed, the dual of $\und{\vectdual{E}}$ is naturally
isomorphic to the model vector bundle $\Vmodel(A)$.

Let us denote by $(x^a)$, $1\leq a\leq m$, the local coordinates on $M$, by $(y^i)$, $0\leq i\leq n+1$, the
coordinates in the fibers of $E\to M$, so that $A\subset E$ is described by the equation $y^{n+1} = 1$ and
$y^i(v_A) = \zd_{0}^i$. Denote the conjugate momenta in $\ctgT E$ by $p_a$ and $\zp_i$. The action $\zc_{st}$
of $(s, t)\in \Real\times \Real$ reads as
$$
\zc_{st}^*(y^0) =y^0+s, \quad \zc_{st}^*(\zp_{n+1}) = \zp_{n+1} +t
$$
and $\zc_{st}^*(y^i) = y^i$, $\zc_{st}^*(\zp_j) = \zp_j$ for $i\neq 0$ and $j\neq{n+1}$. The structure of a
double vector bundle on $\ctgT E$ is given by the two commuting homoteties $h_t^i:\ctgT E\to \ctgT E$,
$$
h_t^1(x^a, y^i, p_a, \zp_i) = (x^a, y^i, t\cdot p_a, t\cdot \zp_i), \quad h_t^2(x^a, y^i, p_a, \zp_i) = (x^a,
t\cdot y^i, t\cdot p_a, \zp_i), \quad t\in \Real.
$$
Note that if $\zw$ and $\zw'$ are in the same orbit of $\zc$ then the same holds for $h_t^i(\zw)$ and
$h_t^i(\zw')$, $i=1,2$. Hence we get  well-defined commuting homogeneous structures $\tilde{h}^1$,
$\tilde{h}^2$ on the quotient space $\ctgT E/\Real\times\Real$. This proves the first assertion of the
theorem.

Let us compute images of the projections $\tilde{h}^1_0$ and  $\tilde{h}^2_0$. The image of $\tilde{h}^1_0$
consists of null $1$-forms on $E$ modulo the translations in the direction of $v_A$, hence
$$
\tilde{h}^1_0 (\affctg{A}) \simeq E/\langle v_A\rangle = \und{E}.
$$
The image of $\tilde{h}^2_0$ consists of orbits of $\zw_v \in \ctgT_v E$ with $v\in M\subset E$ such that
${\zw_v}_{|\tgT_vM}$ is zero, hence the image is clearly isomorphic with $\vectdual{E}/\langle \za_A\rangle
\simeq \und{\vectdual{E}}$. Here we identified $M$ with  the image of the zero section of $E\to M$.

Let us define the linear functions $l_1: \und{E}\to \Real$ and $l_2:\und{\vectdual{E}}\to \Real$  by
\be{e:lifts}
l_1(v + \Real\cdot v_A(m)) := \za_A(v), \quad l_2(f+\Real\cdot \za_A):= f(v_A(m)),
\ee
for $v\in E_m$, $f\in\Aff(A_m, \Real)\simeq \vectdual{E_m}$. The functions $l_1$, $l_2$ are well defined
because $\za_A(v_A) = 0$. The pullbacks of $l_i$, $i=1,2$, $\tilde{l}_i: \affctg{A} \to \Real$ are
\be{e:lifts2}
\tilde{l}_1([\zw_v]) = \za_A(v), \quad \tilde{l}_2([\zw_v]) = \langle \zw_v, (X_A)_v \rangle,
\ee
where $[\zw_v]$ is the class of $\zw_v\in \ctgT_v E$ in $\affctg{A}$ and the  vector field $X_A\in\Vect(E)$ is
the vertical lift of $v_A\in\Sec(E)$. Note that the subset of $E$ given by the equation $l_1=1$ is $\und{A}$
and the subset of $\und{\vectdual{E}}$ given by $l_2=1$ is $\und{\affdual{A}}$.

Let us recall that the vector hull of $\zz:\Phase A\to\und{A}$ is isomorphic with the {\it reduced} cotangent
bundle $\redctgT A$ (\cite{GGU,GGU1,GGU2}), the orbit space of the $(\Real, +)$-action on $\ctgT A$ induced by
translations along $v_A$. Although there is no canonical embedding of $\ctgT A$ into $\ctgT E$, we shall show
that we do have such an embedding of $\redctgT A$ into the factor space $\affctg{A}$. Indeed, $\ctgT A \simeq
(\ctgT E)_{|A}/\Real\cdot d\za_A$, hence
$$
\redctgT A \simeq (\ctgT E)_{|A}/\Real\times\Real\hookrightarrow (\ctgT E)/\Real\times\Real = \affctg{A}.
$$

\noindent Within this embedding, $\redctgT A$ is given in $\affctg{A}$ by the equation $\tilde{l}_1=1$.
Moreover, it is known (see \cite{GGU2}) that $\Phase A$ can be naturally identified with the following affine
subbundle of $\redctgT A$:
\be{e:PAinTstarA}
\Phase A = \{[\zw_v]\in \redctgT A: \langle \zw_v, (X_A)_v \rangle = 1 \},
\ee
hence $\Phase A$ is given in $\affctg{A}$ as
$$
\Phase A =  \{ [\zw]\in \affctg{A}: \tilde{l}_1([\zw]) = 1 = \tilde{l}_2([\zw]) \}.
$$
We claim that $\Phase A$ is isomorphic to the double affine bundle associated with the double vector bundle
$\affctg{A}$ and the functions $l_1$ and $l_2$. One easily checks that the projections from $\affctg{A}$ onto
$\und{E}$ and $\und{\vectdual{E}}$ correspond to the projections from $\Phase A$ onto $\und{A}$ and
$\und{\affdual{A}}$, respectively,  and moreover, the affine combinations in $\Phase A$ are compatible with
the homogeneous structures $\tilde{h}^1$, $\tilde{h}^2$ on $\affctg{A}$, what proves our claim. The
isomorphism $\affctg{A} \simeq \dblvhull{\Phase A}$ follows now from Proposition \ref{prop: isohulldvb}.

\epf

\begin{prop}\label{prop:modelledPA} The model double vector bundle of the double phase bundle $\Phase A$
is canonically isomorphic to the double vector bundle $(\ctgT F; F, \vectdual{F}; M)$ with
$F=\Vmodel(\und{A})$.
\end{prop}
\bepf From Proposition \ref{prop: isohulldvb},
the model double vector bundle $\dblVmodel(\Phase A)$ is recognized in $\affctg{A}$ as the set of orbits
$[\zw]\in \affctg{A}$ such that $\tilde{l}_1([\zw])= 0 = \tilde{l}_2([\zw])$, i.e.
$$
\dblVmodel(\Phase A) = \{[\zw_v]\in \affctg{A}:  \zw_v\in \ctgT_v E, v\in \Vmodel(A)\subset E, \langle\zw_v,
(X_A)_v\rangle = 0 \}.
$$
We want to find a canonical isomorphism $\dblVmodel(\Phase A)\simeq \ctgT \Vmodel(\und{A})$. Note that
$\Vmodel(\und{A})\simeq \und{\Vmodel(A)}$, naturally. We shall find an injection $\imath: \ctgT
\und{\Vmodel(A)}\to \affctg{A}$. It can be defined as the following composition
$$
\xymatrix{\ctgT \und{\Vmodel(A)} \ar@{-|>}[r] & \ctgT \Vmodel(A)\simeq (\ctgT E)_{|\Vmodel(A)}/\Real \ar[r] &
(\ctgT E)_{|\Vmodel(A)}/\Real\times\Real \ar@{^(->} @<-0.4ex>[r]& \affctg{A}\simeq \ctgT E/\Real\times \Real
}\,,
$$
where $\xymatrix{\ctgT \und{\Vmodel(A)} \ar@{-|>}[r]& \ctgT \Vmodel(A)}$ is the induced pullback relation with
respect to the bundle projection $\Vmodel(\zz):\Vmodel(A)\to\und{\Vmodel(A)}$. It is easy to check  that
$\imath$ is an injective function, preserves the homogeneous structures and gives an isomorphism of $\ctgT
\und{\Vmodel(A)}$ with the image of $\imath$ which is clearly $\dblVmodel(\Phase A)$.

\epf

The core of the double affine bundle $\Phase A$ is $\ctgT M$. To put a special double affine structure on
$\Phase A$, suppose that there exist a nowhere vanishing $1$-form $\zw_M\in \zW^1(M)$ on $M$. We are going to
investigate the duals of the special double affine bundle $(\Phase A, \zw_M)$. For, we need first to describe
the dual of the vector bundle $\affctg{A}\to \und{E}$. Let us consider the subbundle
$$
\tgT^{\text{hor}}E := \{X_v\in \tgT_v E: \langle d_v\za_A, X_v\rangle = 0, v\in E\}
$$
of $\tgT E$. We call $\tgT^{\text{hor}}E$ the set of {\it horizontal} vectors. The $(\Real, +)$-action $\zvf$
on $E$ induced by translations along $v_A$ can be lifted to an action on $\tgT E$. It moves horizontal vectors
into horizontal ones, so we can consider the orbit space of this action restricted to $\tgT^{\text{hor}}E$:
$$
\afftg{A} := \{[X]: X\in \tgT^{\text{hor}}E \}\,,
$$
where $[X]$ denotes the orbit of $X\in \tgT^{\text{hor}}E$ of $(\Real, +)$-action on $\tgT E$ induced by
translations along $v_A$.

\begin{prop}\label{prop:modelledPA1} The vertical dual of the double vector bundle
$(\affctg{A}; \und{E}, \und{\vectdual{E}}; M)$  is canonically  isomorphic to
$$
\xymatrix{\afftg{A} \ar[r] \ar[d] & \tgT M \ar[d] \\
\und{E}\ar[r] & M }
$$
Its core bundle is the model vector bundle $\Vmodel(A)$ of $A$.
\end{prop}
\bepf
The natural pairing $\langle \cdot, \cdot \rangle: \ctgT E\times \tgT E\to \Real$ satisfies
$$
\langle d_v\za_A, X_v\rangle = 0
$$
for $X_v\in \tgT_v^{\text{hor}}E\subset \tgT_v E$ and is invariant with respect to $(\Real, +)$-actions on
$\tgT E$ and $\ctgT E$ induced by translations along $v_A$. Thus it induces a well-defined and non-degenerate
pairing $\affctg{A}\times \afftg{A}\to \Real$. It is also clear that the vector bundle structures $\tgT E\to
E$ and $\tgT E\to \tgT M$ can be passed to, respectively, $\afftg{A}\to \und{E}$ and $\afftg{A}\to \tgT M$. If
$X_v\in \tgT_v^{\text{hor}}E$ lies in the intersections of the kernels of the morphisms $\tgT^{\text{hor}}E\to
\tgT M$ and $\tgT^{\text{hor}}E\to\und{E}$, then $v$ lies in the line spanned by $v_A(m)$, for some $m\in M$
and $X_v$ is tangent to a fiber of $\Vmodel(A)\to M$. The  isomorphism of the core bundle with $\Vmodel(A)$
comes from the identification of the vector bundle $E$ with the subbundle $(VE)_{|M}\subset \tgT E$ of
vertical vectors, restricted to $M\subset E$.

\epf

By virtue of  Theorem \ref{thm:dual2}, the vertical dual of the special double affine bundle $(\Phase A,
\zw_M)$ is a double affine subbundle of $\afftg{A}$ described by means of linear functions $l_1$ and $l_3$,
where $l_1$ is given in (\ref{e:lifts}) and $l_3$ is the linear function on $\tgT M$ associated with $\zw_M$.
The subset of $\afftg{A}$ described by the pullbacks of linear functions $l_1$, $l_3$ is equal to
$$
\dualP A := \{ [X_v]: v\in A\subset E, X_v\in \tgT_v A \subset \tgT_vE, \langle (\zh^*\zw_M)_v, X_v \rangle =
1 \},
$$
where $\zh^*\zw_M$ is the pullback of $\zw_M$ to a $1$-form on $A$. The set $\dualP A$ is an affine subbundle
of the {\it reduced} tangent bundle $\tilde{\tgT}A\to\und{A}$ consisting of those orbits $[X_a]\in
\tilde{\tgT}A$, $X_a\in \tgT_aA$, such that $\langle (\zh^*\zw_M)_a, X_a\rangle =1$. The dual bundle to
$\und{\vectdual{E}}$ is clearly $\Vmodel(A)$ and the linear function $l_2$ on $\und{\vectdual{E}}$ corresponds
to $v_A\in\Sec(\Vmodel(A))$. We summarize the above discussion as follows.
\begin{prop}\label{prop:dualPA}
The vertical and horizontal duals of the special double affine bundle $(\Phase A, \zw_M)$ are naturally
isomorphic to
$$
\xymatrix{\dualP A \ar[r] \ar[d] & \subaff{(\tgT M)}{\zw_M } \ar[d]  &  &
\dualP \affdual{A} \ar[r] \ar[d] & \und{\affdual{A}} \ar[d]\\
\und{A}\ar[r] & M &  \text{\rm and}  &  \subaff{(\tgT M)}{-\zw_M }\ar[r] & M\,, }
$$
respectively, where $\subaff{(\tgT M)}{\zw_M } = \{X\in T M: \langle \zw_M, X\rangle = 1\}$ is an affine
subbundle of $\tgT M$. The core bundles are the special vector bundles ($\Vmodel(A), v_A)$ and
$(\Vmodel(\affdual{A}), 1_A)$, respectively.
\end{prop}
\bepf
The form of the vertical dual of $(\Phase A, \zw_M)$ is already justified. Let us considered  the canonical
isomorphism $\zb_E$ from $\ctgT E$ to (the flip of) $\ctgT \vectdual{E}$ of double vector bundles, which in
the adapted coordinate systems $(x_a; y_i; p_b; \zp_j)$ and $(x_a; \zx_i; q_b; z_j)$ on $\ctgT E$ and $\ctgT
\vectdual{E}$, respectively,  has the following form
\be{e:TulczyjewIso}
\zb_E^*(x_a)=x_a, \quad \zb_E^*(\zx_i)=\zp_i, \quad \zb_E^*(q_b) = -p_b, \quad \zb_E^*(z_j) = y_j.
\ee
It has been first discovered by Tulczyjew for $E = \tgT M$ (\cite{Tu}). It is the identity on the side bundles
$E$, $\vectdual{E}$, and minus the identity on the core $\ctgT M$. It induces an isomorphism of the double
vector bundles $\affctg{A}$ and the flip of $\affctg{\affdual{A}}$, and also between the special double affine
subbundles $(\Phase A, \zw_M)$ and  the flip of $(\Phase \affdual{A}, -\zw_M)$. What we need now is the flip
of the vertical dual of $(\Phase \affdual{A}, -\zw_M)$, the form of which we derive from the diagram on the
left in the proposition.

\epf

We shall visualize the theorems from this section in two  diagrams for which we shall give now some
explanation. Let us first assume that $f: Z_1 \to Z_2$ is a morphisms between AV-bundles with the base map
$\und{f}:M_1\to M_2$. We have an induced relation, called the {\it phase lift} of $f$ and denoted by
$\xymatrix{\Phase f: \Phase Z_2 \ar@{-|>}[r]& \Phase Z_1}$. It consists of all pairs
$(\affdiff_{\und{f}(m)}\za_2, \affdiff_m\za_1)\in \Phase Z_2\times \Phase Z_1$ such that $f\circ \za_1 =
\za_2\circ \und{f}$, where $\za_i\in\Sec(Z_i)$, $i=1,2$. In case of trivial bundles, $Z_i=M_i\times\Real$, the
relation $\Phase f$ coincides with the phase lift $\xymatrix{\ctgT f: \ctgT M_2 \ar@{-|>}[r]& \ctgT M_1}$. We
shall apply the phase "functor" to the following sequence of AV-bundles
$$
\xymatrix{
A\ar[d] &  A\times\Real \ar[d]\ar@<4ex>[l]_q\ar@<-4ex>[r]^\iota & E\times \Real\ar[d] \\
 \und{A} & A  &  E
 }
$$
The right arrow $\iota$ is the injection morphism of trivial bundles. The left arrow $q$ is the affine bundle
morphism, $q(a, r) = a- r \cdot v_A(m)$, for $r\in \Real$ and $a\in A$ lying over $m\in M$. Since
$\Phase(A\times\Real)\simeq \ctgT A$, we get the following sequence of phase bundles and relations between
them:
\be{e:TstarEtoPA}
\xymatrix{\ctgT E \ar@{-|>}[r]^{\Phase\iota}& \ctgT A\ar@{-|>}[r]^{(\Phase q)^{-1}}& \Phase A}
\ee
A covector $\zb_v\in \ctgT_v E$ is in the relation $\Phase\iota$ with a covector $\za_a\in \ctgT_a A$ if and
only if $v\in A\subset E$, $a=v$, and $\zb_{a|\tgT_a A}=\za_a$. The relation $\xymatrix{\Phase q: \Phase
A\ar@{-|>}[r]& \ctgT A}$ consists of pairs $(\affdiff_{\und{a}}\zs, d_a f_\zs)$, where $\zz(a)=\und{a}$,
$\zs\in\Sec_{\und{A}}(A)$, and $f_\zs$ is given by (\ref{e:ProjOnDualA}). Hence $\Phase A$ is realized in
$\redctgT A$ as in (\ref{e:PAinTstarA}). The sequence (\ref{e:TstarEtoPA}) can be seen as a process of
generalized reduction. One can also pass from $\ctgT E$ to $\Phase A$ but with the first step being the
reduction of $\ctgT E$ to $\affctg{A}$. Then we can consider $\Phase A$ as a reduction of $\affctg A$. Note
that the induced relations preserve also the second canonical (affine or vector bundle) structure. Thus we get
a sequence of generalized morphisms (relations) of double affine bundles as it is depicted in the first
diagram below. The dual picture of  this diagram is given on the second one.

$$
\xymatrix{
&   &   &         & \affctg{A} \ar[ld]\ar[rdd]\ar@/^1pc/ @{-|>}^{(2,0)}[rrrddd]    &      &     &     &   \\
&   &   & \und{E}\ar@{-|>}^(0.2){(1, 0)}[rrrddd]  &                                &      &     &     &   \\
&   &   &         &                                & \und{\vectdual{E}} \ar@{-|>}_(0.7){(1, 0)} [rrrddd] &&     &   \\
& \ctgT E\ar@{-|>}^{(1,1)}[rrr]\ar[rdd]\ar[ld] \ar@/^1pc/_(0.3){(0,2)}[rrruuu] &   &  &  \ctgT A \ar[rdd]\ar[ld]
\ar@{-|>}^{(1,1)} [rrr] & &&  \Phase A\ar[ld]\ar[rdd]   &   \\
E\ar@/^3pc/^{(0,1)} [rrruuu] \ar @{-|>}^(0.3){(1,0)} [rrr] & &    & A\ar^(0.3){(0,1)}[rrr]  &  &   & \und{A} & & \\
& & \vectdual{E} \ar^{(0,1)}[rrr] \ar@/^3pc/^(0.6){(0,1)} [rrruuu] &  &  &  \und{\vectdual{E}} \ar^{(1, 0)}
@{-|>} [rrr]  & & & \und{\affdual{A}}
 }
$$

$$
\xymatrix{
&   &   &         & \afftg{A} \ar[ld]\ar[rdd]\ar@/^1pc/ @{-|>}^{(2,0)}[rrrddd]    &      &     &     &   \\
&   &   & \und{E}\ar@{-|>}^(0.2){(1, 0)}[rrrddd]  &                                &      &     &     &   \\
&   &   &         &                                & \tgT M \ar@{-|>}_(0.7){(1, 0)} [rrrddd] &&     &   \\
& \tgT E\ar@{-|>}^{(2,0)}[rrr]\ar[rdd]\ar[ld] \ar@/^1pc/ @{-|>}_{(1,1)}[rrruuu] &   &  &
\tgT A \ar[rdd]\ar[ld] \ar@{-|>}^{(1,1)} [rrr] & &&  \dualP A\ar[ld]\ar[rdd]   &   \\
E\ar@/^3pc/^{(0,1)} [rrruuu] \ar @{-|>}^(0.3){(1,0)} [rrr] & &    & A\ar^(0.3){(0,1)}[rrr]  &  &   & \und{A} & & \\
& & \tgT M \ar[rrr] \ar@/^3pc/ [rrruuu] &  &  &  \tgT M \ar @{-|>}^{(1,0)} [rrr]  & & & \subaff{(\tgT
M)}{\zw_M }
 }
$$
We can here distinguish two types of relations: injections (or inverses of injections) and projections. The
relation $\xymatrix{\ctgT E \ar@{-|>}[r]& \ctgT A}$ is a composition of the inverse of the injection $(\ctgT
E)_{|A}\to \ctgT E$ and the projection onto $(\ctgT E)_{|A}/\Real\simeq \ctgT A$, so it is of mixed type.
Similarly for the reduction $\xymatrix{\ctgT A\ar@{-|>}[r]& \Phase A}$. The reduction $\xymatrix{\ctgT E
\ar[r] & \affctg{A} \ar@{-|>}[r] & \Phase A}$ is the composition of the projection and the inverse of the
injection of $\Phase A$ into $\affctg{A}$. We have put labels $(i, j)$ on some of the arrows in the diagrams
to indicate how many injections $(i)$ and projections $(j)$ it involves.

\section{Contact bundles and double affine duals}

In the previous section we considered the double affine phase bundle $\Phase A$ in the category of special
double affine bundles. In order to do this we needed to distinguish a section of the core bundle which in our
case meant to choose a $1$-form  $\zw_M$ on the base manifold $M$. However, there are purely canonical
examples of special affine bundles. The affine {\it contact bundle} $\Contact A$ (\cite{GGU2}) is such a one.
Let us describe its double affine structure and the duals of $\Contact A$.

Let us first recall what the affine contact bundle is. Assume that $\zz:(Z, 1_M)\to M$ is an $AV$-bundle. A
first-jet of a section $\zs\in\Sec(\zz)$ at $m\in M$ is the class of $\zs$ subject to the following
equivalence relation:
$$
\zs\sim\zs' \quad\text{if and only if}\quad \zs(m)=\zs'(m) \, \text{and}\, d_m(\zs-\zs')=0.
$$
We shall denote such a class by $\cdiff_m\zs$. Note that $\zs-\zs'$ is a section of the model bundle
$M\times\Real\to M$ and so  it can be seen as a function on $M$, hence the differential of $\zs-\zs'$ makes
sense. The affine contact space $\Contact Z$ is a collection of $\cdiff_m\zs$ as $m$ varies through $M$ and
$\zs\in Sec(Z)$. It is an affine bundle over $Z$ (modelled on $\ctgT M$) and an $AV$-bundle over $\Phase Z$:
$$
\xymatrix{\Contact Z \ar[r]^{\zz_{\Contact Z}} \ar[d]^{\zm} & \Phase Z \ar[d]^{\Phase\zz} \\
Z \ar[r]^{\zz} & M }.
$$
It is also an affine bundle over $M$ modelled on $\ctgT M \oplus \Real$, with the affine structure defined by
\be{e:CAaffstr}
\cdiff_m \zs_2 - \cdiff_m \zs_1 = (d_m(\zs_2-\zs_1), \zs_2(m)- \zs_1(m)).
\ee
In other words, $\Contact A$ is the first-jet bundle of $\zz$ (\cite{S}).

Now let $(\zh, v_A)$, $\zh:A\to M$, $v_A\in \Sec(\zh)$, be a special affine bundle. The affine contact bundle
of $A$, denoted by $\Contact \zz: \Contact A\to\und{A}$, is by definition the affine contact bundle of the
$AV$-bundle $\zz:A\to\und{A}$. The contact manifold $\Contact A$ is naturally fibred over $A$, and  also  over
$\affdual{A}$ with the projection $\Contact A\to \affdual{A}$ given by $\cdiff_{\und{a}}\zs \mapsto f_\zs$,
where  $f_\zs \in \affdual{A}_m$ is defined in (\ref{e:ProjOnDualA}), and $m=\zh(a)$. The image of the
projection $\Contact A\to A\times_M\affdual{A}$ is $\{(a, f): a\in A_m, f\in\affdual{A}_m, m\in M, f(a) =
0\}$, hence the condition $(ii)$ in the definition  of a double affine bundle is not satisfied. However, we do
get a special double affine bundle structure on $\Contact A$ but fibred over $\und{A}$ and
$\und{\affdual{A}}$:
\be{d:CA}
\xymatrix{\Contact A \ar[r]^{\affdual{\Contact }\zz} \ar[d]^{\Contact \zz} & \und{\affdual{A}}
\ar[d]^{\und{\affdual{\zh}}} \\
\und{A}\ar[r]^{\und{\zh}} & M }.
\ee
Since the core of the affine double bundle $\Phase A$ is isomorphic to $\ctgT M$, the core bundle of $\Contact
A$ is the special vector bundle $(\ctgT M\oplus\Real, (0,1_M))$. There is a canonical isomorphism $\zk:
\Contact A\to \Contact \affdual{A}$ (\cite{U}), which is a morphism of  double affine bundles. It is identity
on the bases $\und{A}$, $\und{\affdual{A}}$ and induces also isomorphism of $\Phase A$ and $\Phase\affdual{A}$
of double affine bundles:
\be{d:LCA}
\xymatrix{ & \Contact A \ar[d] \ar[rrr]^{\zk} &&& \Contact \affdual{A}\ar[d] &  \\
 & \Phase A \ar[dl] \ar[dr] \ar[rrr] &&& \Phase\affdual{A}\ar[dl] \ar[dr] &  \\
\und{A}\ar[dr] \ar@/^1pc/^(0.3){id}[rrr] &&   \und{\affdual{A}}\ar[dl] \ar@/_1pc/^(0.7){id}[rrr] &
\und{A}\ar[dr] & & \und{\affdual{A}}\ar[dl] \\
& M \ar[rrr]^{id} &&& M & }
\ee

Let us analyze the vertical dual of (\ref{d:CA}). For any $AV$-bundle $\zz:Z\to M$, the affine dual of
$\Contact \zz:\Contact A\to M$ is $\redtgTbar \zz : \redtgTbar Z\to M$, where by definition (\cite{U})
$\redtgTbar Z$ is the space of orbits of the following diagonal action $\zD$ on $\tgT Z$:
$$
(t, v) \mapsto (\zvf_t)_* (v + t\cdot X_Z(z)),
$$
for $v\in \tgT_zZ$, $z\in Z$, $t\in \Real$, where $\zvf_t: Z\to Z$, $z\to z+t$ is the canonical action of
$\Real$ on $Z$, and $X_Z\in\Sec_Z(TZ)$ is the fundamental vector field of this action. In the adapted
coordinate system $(x_a, s; \dot{x}_a, \dot{s})$ on $TZ$ we have $X_Z = -\partial_s$ and the diagonal actions
read as
$$
(t, (x_a, s;\dot{x}_a, \dot{s}))\mapsto (x_a, s+t; \dot{x}_a, \dot{s}-t).
$$
The model vector bundle of $\redtgTbar\zz$ is the reduced tangent bundle $\redtgT \zz: \redtgT Z\to M$. Since
the vector field $X_Z$ is invariant it can be seen as a section of $\redtgT \zz$. We shall consider
$\redtgTbar Z$ as a spacial affine bundle over $M$ with the distinguished section given by $X_Z$. Let us treat
$\Contact Z$ as an affine hyperbundle of $\ctgT Z$ under the canonical embedding $\cdiff_m\zs \mapsto \zw \in
\ctgT_{z} Z$, defined by $\langle \zw, v \rangle X_Z(z) =v-v'$,  where $z= \zs(m)$ and $v'\in \tgT_zZ$ is the
vertical projection of $v\in \tgT_zZ$ onto the tangent space to the image of $\zs$. Then
$$
\Contact Z \simeq \{\zw_z\in \ctgT Z : \langle \zw_z, X_Z(z)\rangle = 1, z\in Z \}.
$$
We have an obvious bi-affine special pairing $\Contact Z \times \redtgTbar Z \to \Real$ induced by the one of
cotangent and tangent bundles, which gives the desired isomorphism $\affdual{(\Contact Z)} \simeq \redtgTbar
Z$ of special affine bundles over $M$. One can prove that $\redtgTbar Z$ is the {\it bundle of affine
derivations} on $Z$ with values in $Z$ denoted by $\Aff_M(\Phase Z, Z)$ (\cite{GGU2}).

Now we apply the above observations to the AV-bundle $\zz:A\to\und{A}$ and find that, according to Theorem
\ref{thm:dual}, the vertical dual of (\ref{d:CA}) is
\be{e:TangentDual}
\xymatrix{\redtgTbar A \ar[r]^{\redtgTbar\zh} \ar[d]^{\zt} & \tgT M \ar[d]^{\zt_M} \\
 \und{A}\ar[r]^{\und{\zh}} & M }
\ee
since $\ahyp{D}_3  = \ahyp{(\ctgT M\oplus \Real)} = \tgT M$. Note that the map $T\zh: \tgT A \to \tgT M$
factors to a map from $\redtgTbar A$ which is an affine bundle over $TM$. The core of (\ref{e:TangentDual}) is
$(\und{\affdual{A}})^{\dagger} \simeq \Vmodel(A)$ which is a special vector bundle. Indeed, as
$\Aff(\affdual{A}, \Real) = (\affdual{A})^{\dagger} \simeq \vhull{A}$, the elements of
$(\und{\affdual{A}})^{\dagger}$ correspond to the vectors $v\in \vhull{A}$ such that the special element $1_A$
of $\vectdual{(\vhull{A})}$ annihilates $v$, and so $v\in\Vmodel(A)$. The described isomorphism preserve also
the distinguished elements: the constant function equal $1$ on $\und{\affdual{A}}$ and $v_A\in
\Sec(\Vmodel(A))$. This way we have discovered another canonical object in the category of special double
affine bundles.

\medskip
In a vector bundle setting, the cotangent bundle $\ctgT E$ is a very interesting and intriguing object. Recall
that, if $E$ is an $n$-vector bundle and $\zD_i$, $i=1, \ldots, n$, are the  Euler vector fields corresponding
to the $n$-vector bundle structures on $E$, then $\ctgT E$ is canonically an $(n+1)$-vector bundle, whose
vector bundle structures are encoded in the Euler vector fields $\ctgT \zD_i$, which is by definition the
phase lift of $\zD_i$, and the natural cotangent vector bundle structure. Moreover, the canonical symplectic
structure on $\ctgT E$ gives the pairings between $E$ and the other bases of the  side bundles of $\ctgT E$
(\cite{GR}). In the following we are going to find a similar passage but in an affine setting.

Let $(\zh:A\to M, v_A)$ be a special affine bundle, $E:=\vhull{A}$ be its vector hull and let
$\za_A\in\Sec(\vectdual{E})$ be the distinguished section. We shall consider $\za_A$ and $v_A$ as linear
functions on $E$ and $\vectdual{E}$, respectively, and then as the functions on $\ctgT E$ via pullbacks with
respect to the canonical projections $\ctgT E\to E$ and $\ctgT E\to \vectdual{E}$. Let us define $\dblB A$ as
a double affine subbundle $(\ctgT E; E, \vectdual{E}; M)$ determined by the following  linear functions on $E$
and $\vectdual{E}$: $\za_A$ and (evaluation on) $v_A$. One easily finds that
$$
\dblB A = \{\zw_x\in \ctgT_x E: x\in A\subset E, \langle\zw_x, X_A(x)\rangle = 1\},
$$
where $X_A\in\Sec_E(\tgT E)$ is the vertical lift of $v_A$. As $\dblB A$ is a canonical double affine bundle
with side bundles being dual special affine bundles $A$ and $A^\#$, we will call it the {\it double affine
dual bundle} of $A$ (or $A^\#$).

Note that $\dblB A$ can also be described as the set of first jets $j^1_{\und{a}}\zs$ of the bundle $E\to
\und{E}$ at points $\und{a}\in \und{A}\subset \und{E}$. It is also possible to present $\dblB A$ as a middle
step of the reduction from $\ctgT E$ to $\Phase A$:
$$
\xymatrix{& \affctg A \ar@{-|>}[rd]^{(2, 0)} &  \\
\ctgT E \ar@{-|>}[r]^{(1, 1)}  \ar[ru]^{(0,2)}\ar@{-|>}[rd]_{(2, 0)} & \ctgT A\ar@{-|>}[r]^{(1, 1)}& \Phase A \\
&  \dblB A \ar [ru]_{(0,2)}& }
$$
The projections in the diagram
$$
\xymatrix{\dblB A \ar[r]^{\zp_{\affdual{A}}} \ar[d]^{\zp_A} & \affdual{A} \ar[d] \\
A \ar[r] & M }
$$
read as $\zp_A(\zw_x) =x$ for $\zw_x\in \ctgT_x E\cap\dblB A$, $\zp_{\affdual{A}}(\zw_x):A_m\to\Real$ is the
composition
$$
A_m \hookrightarrow E_m \simeq \tgT_x E_m \stackrel{\zw_x}{\rightarrow} \Real,
$$
where $x\in E_m$ and $\zw_x$ is restricted to the tangent space to the fiber $E_m$. Let $(y_i)$ and $(\zx_i)$
be local bases  of sections of $\vectdual{E}$ and $E$, respectively, such that $y_i(\zx_j) = \zd_{ij}$. We
shall consider them also as linear functions on dual bundles. Let $\zb_E: \ctgT E\to \ctgT \vectdual{E}$ be
the Tulczyjew isomorphism (\ref{e:TulczyjewIso}), which in the adapted coordinate systems $(x_a; y_i; p_b;
\zp_j)$ and $(x_a; \zx_i; q_b; z_j)$ on $\ctgT E$ and $\ctgT \vectdual{E}$, respectively, reads as $\zb_E(x_a;
y_i; p_b; \zp_j)=(x_a; \zp_i; -p_b; y_j)$. The restriction of $\zb_E$ to $\dblB A$ induces an isomorphism of
double affine bundles $\dblB A$ and (the flip of) $\dblB \affdual{A}$:
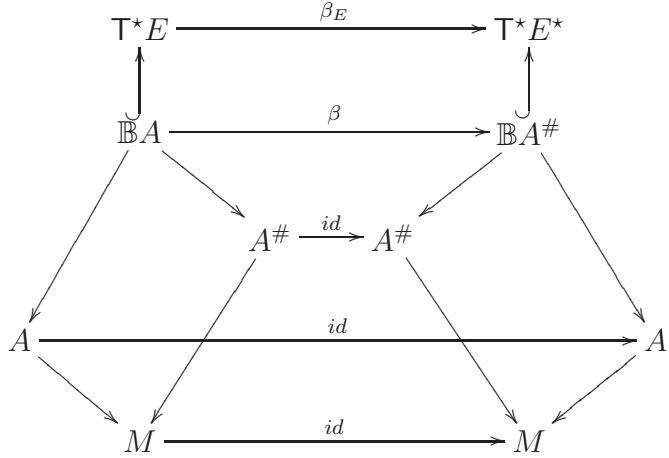
$$
\xymatrix{
& \ctgT E     \ar[rrr]^{\zb_E} &&& \ctgT \vectdual{E} & \\
& \dblB A  \ar@{^(->}[u]\ar[ldd]\ar[rd]\ar[rrr]^\zb      &&& \dblB \affdual{A}\ar@{^(->}[u]\ar[ld]\ar[rdd] & \\
&         & \affdual{A}\ar[ldd] \ar[r]^{id} &  \affdual{A}\ar[rdd]  &  &  \\
A \ar[rd]\ar[rrrrr]^{id} &                  &    &  & &  A\ar[ld] \\
& M\ar[rrr]^{id} &&& M & }
$$
The space $\dblB A$ is preserved by the action $\zc = (\zc_1, \zc_2)$ (defined in (\ref{e:action1}),
(\ref{e:action})) of $\Real\times\Real$ on $\ctgT E$. We now consider $\zc=\zc^A$ (respectively,
$\zc^{\affdual{A}}$) as an action on $\dblB A$ (respectively, $\dblB \affdual{A}$). The orbit space of $\zc$,
$\dblB A /{\Real\times\Real}$, is canonically identified with the affine phase bundle $\Phase A\subset
\redctgT A$ (\cite{GGU2}). The isomorphism $\zb_E$ moves the orbits of $\zc^A$ into the orbits of
$\zc^{\affdual{A}}$, and so it induces an isomorphism of the orbit spaces, $\Phase A$ and $\Phase\affdual{A}$.
Moreover, the orbit space of $\zc^A_ 2$, $\dblB A /{\{0\}\times \Real}$, can be naturally identified with the
affine contact bundle $\Contact A$. Indeed, thanks to our convention of the distinguished section of the model
bundle of an AV-bundle, $\Contact A$ is identified with the affine hyperbundle of $\ctgT A$ of those covectors
$\zw\in \ctgT_a A$, $a\in A$, such that $\langle \zw, X_A(a) \rangle =1$, where $X_A$ is the vertical lift of
$v_A$.  Beside, each orbit of $\zc^A_2$ (respectively, $\zc^A_1$) is moved by $\zb_E$ to an orbit of
$\zc^{\affdual{A}}_1$ (respectively, $\zc^{\affdual{A}}_2$).  It turns out that the orbit spaces of $\zc_1$
and $\zc_2$ are canonically isomorphic:

\begin{theo}There is a canonical isomorphism
$$
\zt: \dblB A/{\{0\}\times\Real} \to \dblB A/{\Real\times \{0\}}.
$$
The composition
$$
\Contact A\simeq \dblB A/{\{0\}\times\Real} \stackrel{\zt}{\rightarrow} \dblB A/{\Real\times \{0\}}
\stackrel{\zb}{\rightarrow} \dblB \affdual{A}/{\{0\}\times\Real}\simeq \Contact \affdual{A}
$$
is an isomorphism of double affine bundles. However, it moves the distinguished section $(0, 1_A)$ of the
model vector bundle $\ctgT \und{A}\oplus \Real$ into the section $(0, -1_{\und{\affdual{A}}})$ and so it
establishes an isomorphism of special affine bundles
$$
\Contact \bar{A} \simeq \Contact \affdual{A},
$$
where $\bar{A} = (A, -v_A)$ is the adjoint of $A$.
\end{theo}
\bepf
Let us choose the local basis   $(y_i)$ and its dual $(\zx_i)$ of sections of $\vectdual{E}$ and $E$,
respectively, in such a way that $y_{n+1} = \za_A$ and $\zx_0 = v_A$. Let us define $\zt$ in the adapted
coordinates $(x_a, y_i; p_b; \zp_j)$  on  $\ctgT E$ by the following formula
$$
\zt(x_a; y_0, \ldots, y_n, 1; p_b; 1, \zp_1, \ldots, \zp_n, \Real) = (x_a; \Real, y_1, \ldots, y_n, 1; p_b; 1,
\zp_1, \ldots, \zp_n, -y_0-\sum_{i=1}^ny_i\zp_i),
$$
where the argument of $\zt$ is an orbit of the action $\zc_2$ of an element from $\dblB A$ and similarly the
value of $\zt$ is an orbit of $\zc_1$. Let us check that $\zt$ is defined independently of the choice of
coordinates. For another choice of local sections $(y_i')$ and $(\zx_i')$ we have
$$
\zp_i' = \sum_i a_{ij}\zp_j,
$$
for a matrix $(a_{ij})$ with entries in $C^\infty(M)$ which satisfies: $a_{00}=a_{n+1, n+1} =1$, $a_{0j} = 0$
for $j>0$, $a_{i, n+1} = 0$ for $1\leq i\leq n+1$. These conditions on the matrix $(a_{ij})$ comes from the
fact that the transition maps preserve the special elements $v_A$ and $\za_A$. Let us calculate
$\zp_{n+1}'\circ \zt$. We have
\bea
\zp_{n+1}'\circ\zt &=&  \sum_{j=0}^N a_{n+1, j}\zp_j\circ \zt + \zp_{n+1}\circ \zt =
\sum_{j=0}^n a_{n+1, j}\zp_j  - y_0 - \sum_{i=1}^n y_i\zp_i \nn \\
&=& \sum_{j, k} a_{n+1, j}a^{jk}\zp_k' - \zp_{n+1} - y_0 -
\sum_{i,j,k} a_{ji}y_j'a^{ik}\zp_k' +y_0\zp_0 + y_{n+1}\zp_{n+1} \nn \\
&=& \zp_{n+1} - \sum_{j=0}^{n+1} y_j'\zp_j' = -y_0' - \sum_{j=1}^n y_j'\zp_j',\nn
\eea
thanks to $\zp_0 = 1$ and $y_{n+1}=1$. The other equalities $y_i'\circ \zt = y_i'$ for $i\neq 0$, $\zp_j'\circ
\zt = \zp_j'$ for $j\neq n+1$ and $p_a'\circ\zt =p_a'$ follow immediately. It is also clear that $\zb\circ\zt
: \Contact A \to \Contact \affdual{A}$ has the desired properties. \epf
\section{$n$-affine bundles}
A notion of a double affine bundle has an obvious generalization for a manifold which is equipped with more
than two affine structures. One can view an $n$-tuple affine bundle $A$ ($n$-affine, in short), $n\in\N$, as
an object which is build by means of gluing trivial ones, which by definition are of the form $A_\za :=
U_\za\times K$, where $(U_\za)$ is a covering of the base manifold $M$, $K = \prod_\ze K_\ze$,  the product
over ${\ze \in\{0, 1\}^n}$, $\ze\neq 0^n$, and $K_\ze$ are fixed vector spaces. The manifold $A_\za$ has $n$
natural affine  structures over the bases $A_{\za, i} = U_\za\times\prod_{\ze(i)=0} K_\ze$ and each of
$A_{\za, i}$ is an $(n-1)$-affine bundle. In order to preserve these affine structures, the gluing maps
$\zvf_{\za, \zb}:U_{\za\zb}\times K\to U_{\zb\za}\times K$, where $U_{\za\zb} = U_\za\cap U_\zb$, have to
preserve the "filtration" of the structure sheaf. We shall explain this condition more precisely. Let us
assign the degree $\ze\in\Zet^n$ to linear coordinates on $K_\ze$ and degree $0$ to functions on $U_\za$ and
let ${\cal F}_\za$ be the algebra generated by $C^\infty(U_\za)$ and linear functions on $K_\ze$ as $\ze$
varies in $\{0, 1\}^n$, $\ze\neq 0^n$. The algebra ${\cal F}_\za$ is a subalgebra of smooth functions on
$A_\za$. Let ${\cal F}_{\za, \zm}$ be the subspace of those elements in ${\cal F}_\za$ which are of degree
less or equal $\zm\in\{0, 1\}^n$ with respect to the  product (partial) order on $\Zet^n$. The condition for
preserving the affine structures is that the pullback of a function $f\in {\cal F}_{\zb, \zm }$ restricted to
$U_\za$ with respect to $\zvf_{\za, \zb}$, $f_{|U_\za\times K}\circ \zvf_{\za, \zb}$,  is a function of degree
less or equal $\zm$, for $\zm\in\{0, 1\}^n$. In case $n=2$ we recover that this condition says that gluing
transformations are of the form (\ref{coord}), where the coordinates $y^j$, $z^a$, $c^u$ have been assigned
degrees $(0,1)$, $(1,0)$ and $(1,1)$, respectively. Note that the algebra
$$
{\cal{A}} := \{ f\in C^{\infty}(A): f_{|A_\za}\in {\cal F}_{\za} \}
$$
is not graded but filtered by $\Zet^n$, i.e. for $\zm\in\Zet^n$ there are distinguished subspaces ${\cal
A}_\zm$ of ${\cal A}$ such that ${\cal A}_{\zm_1}\subset {\cal A}_{\zm_2}$ whenever $\zm_1$ is less or equal
$\zm_2$ with respect to the product partial order on $\Zet^n$ and $\bigcup_{\zm} {\cal A}_\zm = {\cal A}$.

Alternatively, one can define an $n$-affine bundle $A$ as a subset of an $n$-vector bundle $E$ determined by
$n$ functions $l_1, \ldots, l_n$, where the degree of $l_i$ is $\ze_i\in \Zet^n$, $\ze_i(j)=\zd_{ij}$, by
setting
$$A = \{x\in E: l_i(x)=1\, \text{for}\, i=1, \ldots , n\},$$
with obvious $n$-affine structures inherited from $E$. Here we consider $E$ as a $\Zet^n$-graded manifold, as
in (\cite{GR}). Its structure sheaf is generated only by functions of degrees from $\{0, 1\}^n$.

Let us assume now that $A$ is an $n$-affine bundle given in an $n$-vector bundle $E$ by linear functions $l_1,
\ldots, l_n$, where the degree of $l_i$ is $\ze_i\in\Zet^n$, $\ze_i(j)=\zd_{ij}$. Let us consider the
$(n+1)$-vector bundle $\ctgT E$ as $\Zet^{n+1}$-graded manifold, as in (\cite{GR}). If $(y_\za^j)$ is the
adapted coordinate system on $E$, where the degree of $y_\za^j$ is $\za\in\Zet^n$, $\za(k)\in\{0, 1\}$, for
$k=1, \ldots, n$, then the induced coordinate system on $\ctgT E$ is of the form $(y_\za^j, p_\za^j)$, where
$p_\za^j$ are the corresponding momenta which have the degree equal to $|p_\za^j| = -\za +1^{n+1} \in
\Zet^{n+1}$, where $1^k \in \Zet^k$  denotes the vector of  ones and we assumed that $\za(n+1)= 0$. To get an
$(n+1)$-affine subbundle of $\ctgT E$ out of the linear functions $l_1, \ldots, l_n$, we need to choose
additionally a linear function of degree $\ze_{n+1}\in\Zet^{n+1}$. Let us recall that for any $n$-vector
bundle $E$, the core $C$ of $E$ (sometimes called ultracore (\cite{Mac})) is defined as $C = \cap_{i=1}^n\ker
\zp_i$, where $\zp_i:E\to E_i$, $i=1, \ldots, n$, are the side bundles of $E$. In the adapted local coordinate
system $(y_\za^j)$ on $E$,  the core $C$  is given by
$$
C=\{x\in E: y_\za^j(x)=0\,\text{for}\, \za\neq1^n\, \text{and}\, \za\neq 0^n\}.
$$
The core $C$ is a vector bundle over $M$, the total base of $E$. There is a canonical action of the core
bundle on the fibration $E\to M$. In the adapted local coordinate system, for $c\in C$ and $v\in E$ lying over
the same point $m\in M$, it reads as
$$
y_\za^j(c+v) := \left\{
\begin{array}{ll}
y_\za^j(c) + y_\za^j(v) & \text{for } \za=1^n, \\
y_\za^j(v) & \text{otherwise.}
\end{array}
\right.
$$
Note that $\zp_j(c+v) = \zp_j(v)$ for any $j=1,2, \ldots, n$ and $v\in A\subset E$  if and only if $c+v \in
A$. The canonical projection $\ctgT E\to \vectdual{C}$ can be defined as the restriction of $\zw : \tgT_vE \to
\Real$ to the tangent space at $v\in E$ of the orbit $v+ C_m$ under the natural identification $\tgT_v(v +
C_m) \simeq C_m$, where $C_m$ is a fiber of $C\to M$ over $m$. A function $l_{n+1}$ on $\ctgT E$ of degree
$\ze_{n+1}$ has a local form $\sum_j f_j p_\za^j$, where $\za = 1^n\in\Zet^n$ and $f_j\in C^\infty(M)$. Hence
$l_{n+1}$ is a pullback of a linear function on $\vectdual{C}\to M$ with respect to the mentioned projection
$\ctgT E\to \vectdual{C}$ and  so it corresponds to  a nowhere vanishing section of the core bundle $C\to M$.
This way we arrived at a definition of a special $n$-affine bundle.
\begin{definition} {\rm A {\it special} $n$-affine bundle is an $n$-affine bundle together with a nowhere
vanishing section of the core bundle.}
\end{definition}
Let $A$ be a special $n$-affine  bundle given in $E$ by means of $l_1, \ldots, l_{n+1}$ as above, where
$l_{n+1}$ determines the special section of the core bundle. We can consider $l_{n+1}$ as a function on the
total space $\ctgT E$ thanks to the canonical projection $\ctgT E\to \vectdual{C}$. We define the {\it
$(n+1)$-affine dual bundle} $\dblB A$ as an $(n+1)$-affine subbundle of $\ctgT E$ by
$$
\dblB A  = \{\zw\in \ctgT E: l_i(\zw)=1, \text{for}\, i=1, \ldots, n+1\}.
$$
Let us denote $A_i = \{x\in E_i: l_j(x)=1 \text{ for } j\neq i, 1\leq j\leq n+1\}$, $i=1, \ldots, n$, the
bases of the side bundles of $A$. The affine bundle $A\to A_i$ is a special corank one subbundle of $E\to
E_i$. The special section is induced from $l_{n+1}$ and the mentioned canonical action of the core bundle
$C\to M$ on $A$. Hence $\dblaffdual{A}{A_i}\to A_i$ is an affine corank $1$ subbundle of $\dblvecdual{E}{E_i}
\to E_i$, which we know is one of the side bundles of $\ctgT E$. Moreover, the total space
$\dblaffdual{A}{A_i}$ has also another $(n-1)$-affine bundle structures, which are denoted by $\aff_j^*$,
$j\neq i$, $1\leq j\leq n$, and are induced from the vector bundle structure of $\ctgT E\to
\dblvecdual{E}{E_j}$, whose Euler vector field is the phase lift of the Euler vector field for $E\to E_j$
(\cite{GR}).
\begin{theo}
The bases of the side bundles of the $(n+1)$-affine bundle $\dblB A$ are $A$ and its duals
$\dblaffdual{A}{A_1}, \ldots, \dblaffdual{A}{A_n}$. Moreover, $(\dblaffdual{A}{A_i}, \aff_j^*)$ is a special
affine bundle which is dual to the adjoint of  $(\dblaffdual{A}{A_j}, \aff_i^*)$, $i\neq j$.
\end{theo}
\begin{proof} The first statement follows from the above discussion.
The second one is an immediate consequence of Proposition \ref{prop:affspduality}, since we can restrict our
consideration to the special double affine bundle $(A, \aff_i, \aff_j)$ with the induced from $l_{n+1}$
section of the core bundle.
\end{proof}

Reassuming, a special $n$-affine bundle $A$ gives rise to an $(n+1)$-affine bundle $\dblB A$. The duals of $A$
can be recognized as the bases of the side bundles of $\dblB A$. The associated pairings are derived from the
canonical symplectic structure on $\ctgT E$, the $(n+1)$-tuple hull of $\dblB A$. We postpone to a separate
paper the problem, how the contact structure on an abstract $n$-affine bundle $A$ determines the pairings
between some $(n-1)$-affine bundles related to the side bundles of $A$ (compare with Theorem 6.1 in
\cite{GR}).

\bigskip
\noindent Janusz Grabowski\\ Institute of Mathematics, Polish Academy of Sciences\\\'Sniadeckich 8, P.O. Box
21, 00-956 Warszawa,
Poland\\{\tt jagrab@impan.pl}\\\\
\noindent Miko\l aj Rotkiewicz\\
Institute of Mathematics,
                University of Warsaw \\
                Banacha 2, 02-097 Warszawa, Poland \\
                 {\tt mrotkiew@mimuw.edu.pl}
\\\\
\noindent Pawe\l\  Urba\'nski\\
Division of Mathematical Methods in Physics \\
                University of Warsaw \\
                Ho\.za 69, 00-681 Warszawa, Poland \\
                 {\tt urba\'nski@fuw.edu.pl}
\end{document}